\documentclass{amsproc}
\usepackage{amssymb}
\usepackage{amsthm}
\usepackage{amsfonts}
\usepackage{multirow}
\usepackage{subfigure}
\usepackage{epsfig}
\usepackage{epstopdf}
\usepackage{array}
\usepackage{hyperref}
\usepackage{verbatim}

\theoremstyle{plain}
\newtheorem{lemma}{Lemma}[section]
\newtheorem{corollary}[lemma]{Corollary}
\newtheorem{definition}[lemma]{Definition}
\newtheorem{example}[lemma]{Example}
\newtheorem{notation}[lemma]{Notation}
\newtheorem{proposition}[lemma]{Proposition}
\newtheorem{remark}[lemma]{Remark}
\newtheorem{theorem}[lemma]{Theorem}
\numberwithin{equation}{section}
\numberwithin{figure}{section}

\begin{document}
\def\arraystretch{2.2}
\def\arraystretch{2.2}
\def\arraystretch{2.2}
\def\arraystretch{2.2}

\title[Dimensions of Non-equicontractive measures]{Local dimensions of
measures of finite type III - Measures that are not equicontractive }
\author[K.~E.~Hare, K.~G.~Hare, G. Simms]{Kathryn E. Hare, Kevin G. Hare,
Grant Simms}
\thanks{Research of K. G. Hare was supported by NSERC Grant RGPIN-2014-03154}
\thanks{Research of K. E. Hare was supported by NSERC Grant 2016 03719}
\thanks{Research of G. Simms was supported by NSERC Grant 2016 03719}
\address{Dept. of Pure Mathematics, University of Waterloo, Waterloo, Ont.,
N2L 3G1, Canada}
\email{kehare@uwaterloo.ca}
\address{Dept. of Pure Mathematics, University of Waterloo, Waterloo, Ont.,
N2L 3G1, Canada}
\email{kghare@uwaterloo.ca}
\address{Dept. of Pure Mathematics, University of Waterloo, Waterloo, Ont.,
N2L 3G1, Canada}
\email{grantjsimms@gmail.com}
\thanks{This paper is in final form and no version of it will be submitted
for publication elsewhere.}

\begin{abstract}
We extend the study of the multifractal analysis of the class of equicontractive
self-similar measures of finite type to the non-equicontractive setting.
Although stronger than the weak separation condition, the finite type property includes examples of IFS that fail the open set condition.

The important combinatorial properties of equicontractive self-similar measures of finite type are extended to the non-equicontractive setting
    and we prove that many of the results from the equicontractive case carry over to this 
    new, more general, setting. In particular, previously it was shown that if an equicontractive self-similar measure of finite type was {\em regular},
    then the calculations of local dimensions were relatively easy.
We modify this definition of regular to define measures to be {\em generalized regular}.
This new definition will include the non-equicontractive case and obtain 
similar results. Examples are studied
of non-equicontractive self-similar generalized regular measures, as well as equicontractive 
    self-similar measures which generalized regular in this new sense, but which are not regular.
\end{abstract}

\maketitle

\section{Introduction}
\label{sec:intro}

Building on the work of Feng \cite{F3, F1, F2}, the first two authors, with various coauthors, studied in \cite{HHM} and \cite{HHN} the multifractal analysis 
of equicontractive, self-similar measures of finite type. This analysis can be
much more complicated than the well understood case when the self-similar
measures satisfy the open set condition. Indeed, the set of local dimensions
of such measures need not be an interval, as was first discovered by Hu and
Lau in \cite{HL} for the special example of the 3-fold convolution of the
classical Cantor measure.  This was further explored by a number of authors, including \cite{FLW, LW, Sh}. In this paper, we extend this theory to non-equicontractive Iterated Function Systems (IFS) of finite type and
their associated self-similar measures. 

The notion of finite type was introduced by Ngai and Wang in \cite{NW}. As they showed, it
includes all IFS satisfying the open set condition where the contraction factors are logarithmically commensurable. 
In addition, finite type also permits some
overlap. We extend the geometric ideas developed by Feng for equicontractive
IFS of finite type, including introducing suitable generalizations of
characteristic vectors, transition matrices and the essential class. Using
these generalizations, the main results of \cite{HHM} are extended to
non-equicontractive, finite type, self-similar measures that are
`comparable'. In particular, the set of local dimensions at points in the
essential class is seen to be a closed interval, and the local dimensions at
the `periodic points' in that class are dense. The essential class always
exists, is unique and is of full measure, both with respect to the 
self-similar measure under investigation, and with respect to the Hausdorff
measure of the support of this self-similar measure. Further, the same formulas hold
for the local dimensions at periodic points as in the equicontractive case (even without comparability).

In order to give criteria that ensures comparability, we also introduce a
notion of `generalized regular' that is defined in terms of the geometry of the IFS and
the probabilities. This property is shown to be more general than the earlier
notion of regular for equicontractive IFS and implies comparability.

The paper is organized as follows. In Section \ref{sec:basic} we develop the geometric
properties of (non-equicontractive) self-similar measures and IFS of finite type. We study the
connection between transition matrices and local dimensions, and establish
the formula for the local dimension of periodic points in Section \ref{sec:local}.
Comparability and generalized regularity are studied in Section \ref{sec:comparable}. The main results on
the structure of the set of attainable local dimensions are also found
there. 

In Section \ref{sec:examples} we consider examples. The (biased) Bernoulli convolution with
contraction the inverse of the golden mean, $r$, arises from the IFS, $%
S_{0}(x)=rx,S_{1}(x)=rx+1-r,$ and probabilities $p, 1-p$. We compare this
measure with the one arising from the (non-equicontractive) IFS, $S_{0}(x)=rx,
$ $S_{1}(x)=1-rx$, and with the same probabilities. The two measures coincide if $%
p=1-p=\frac{1}{2}$. A detailed study is made of their local dimensions when $p<1-p$. In particular,
we show that while both measures admit a closed interval and an isolated
point as their set of attainable local dimensions, the two intervals are
different. Next we give an example of an IFS of finite type
with different (but, positive) contraction factors and an associated,
generalized regular, self-similar measure with the property that the essential class is
the full self-similar set $[0,1]$.  Lastly, we give an equicontractive example
that illustrates our new notion of generalized regular is more encompassing
than the previous notion of regular. 

\section{Basic Definitions and Terminology}
\label{sec:basic}

We begin by introducing the notion of finite type and the related concepts
and terminology that will be used throughout the paper.

\subsection{Iterated function systems and finite type}

By an iterated function system (IFS) we will mean a finite set of
contractions,%
\begin{equation}
S_{j}(x)=r_{j}x+d_{j}:\mathbb{R\rightarrow R}\text{ for }j=0,1,...,k,
\label{IFS}
\end{equation}%
where $k\geq 1$ and $0<$ $\left\vert r_{j}\right\vert <1$. We do not assume $%
r_{j}>0$. When all $r_{j}$ are equal and positive, the IFS is referred to as
equicontractive.

Each IFS generates a unique invariant compact set $K,$ known as its
associated self-similar set, satisfying 
\begin{equation*}
K=\bigcup_{j=0}^{k}S_{j}(K).
\end{equation*}%
By rescaling the $d_{j},$ if necessary, we can assume the convex hull of $K$
is $[0,1]$.

We let $\Sigma $ be the set of all finite words on the alphabet $%
\{0,1,...,k\}$. Given $\sigma =(\sigma _{1},...,\sigma _{j})\in \Sigma $, a
word of length $j$, we will let 
\begin{eqnarray*}
\sigma ^{-} &=&(\sigma _{1},...,\sigma _{j-1}), \\
S_{\sigma } &=&S_{\sigma _{1}}\circ S_{\sigma _{2}}\circ \cdot \cdot \cdot
\circ S_{\sigma _{j}},
\end{eqnarray*}%
and
\begin{equation*}
r_{\sigma }=\prod_{i=1}^{j}r_{\sigma _{i}}.
\end{equation*}%
For each $\sigma \in \Sigma $ , $S_{\sigma }[0,1]$ is a subinterval of $%
[0,1] $ of length $r_{\sigma }$; we call this a \textit{basic interval}. Of
course, in the equicontractive case, all basic intervals arising from words
of length $n$ have length $r^{n}$, where $r$ is the common contraction
factor. In the non-equicontractive case this will not be true. It is
convenient to group basic intervals of about the same length and so we
introduce the following notation: Let 
\begin{equation*}
r_{\min }=\min_{j}\left\vert r_{j}\right\vert
\end{equation*}%
and for $n \ge 1$  put 
\begin{equation*}
\Lambda _{n}=\{\sigma \in \Sigma :\left\vert r_{\sigma }\right\vert \leq
r_{\min }^{n}\text{ and }\left\vert r_{\sigma ^{-}}\right\vert >r_{\min
}^{n}\}.
\end{equation*}%
We let $\Lambda_0$ be the set consisting of only the empty word. One can
check that 
\begin{equation*}
K=\bigcup_{\sigma \in \Lambda _{n}}S_{\sigma }(K)\text{ for each }n.
\end{equation*}

If $\sigma \in \Lambda _{n}$, we will say $\sigma $ is of \textit{generation 
}$n$\textit{\ }and $S_{\sigma }[0,1]$ will be called a \textit{basic
interval of generation }$n$. In the non-equicontractive case there can be
some finite words that are not of any generation.  If $\sigma =(\sigma
_{1},\sigma _{2},...)$ is any infinite word, then for each $n$ there is a
choice of $i_{n}$ ($\geq n)$ such that $(\sigma _{1},...,\sigma _{i_{n}})$
is of generation $n$.

The notion of finite type was introduced by Ngai and Wang in \cite{NW}. The
definition we will use is slightly less general, but is simpler and includes
all the examples in $\mathbb{R}$ that we are aware of.

\begin{definition}
Assume $\{S_{j}\}$ is an IFS as in equation \eqref{IFS}. The words $\sigma ,\tau
\in \Lambda _{n}$ are said to be neighbours if $S_{\sigma }(0,1)\cap S_{\tau
}(0,1)\neq \emptyset $. Denote by $\mathcal{N}(\sigma )$ the set of all
neighbours of $\sigma $. We say that $\sigma \in \Lambda _{n}$ and $\tau \in
\Lambda _{m}$ have the same neighbourhood type if there is a map $f(x)=\pm
r_{\min }^{n-m}x+c$ such that 
\begin{equation*}
\{f\circ S_{\eta }:\eta \in \mathcal{N}(\sigma )\}=\{S_{\nu }:\nu \in 
\mathcal{N}(\tau )\} \ \ \mathrm{and}\ \ f \circ S_\sigma = S_\tau.
\end{equation*}%
The IFS is said to be of \textbf{finite type} if there are only finitely
many neighbourhood types.
\end{definition}

This is equivalent to the condition of finite type given for equicontractive
IFS in \cite{F3}. It was shown in \cite{Ng} that an IFS of finite type
satisfies the weak separation condition, but not necessarily the open set
condition.

\begin{example}
The IFS given by $S_{j}(x)=\pm \rho ^{-n_{j}}x+b_{j}$ where $\rho $ is a
Pisot number, $n_{j}\in \mathbb{N}$ and $b_{j}\in \mathbb{Q[\rho ]}$ was
shown to be of finite type in \cite[Thm. 2.9]{NW}.
\end{example}

\begin{remark}\label{commensurate}
An important fact about iterated function systems of finite type is that the set of contractions
are commensurate, meaning, there are rational numbers $q_{j}$ such that $%
\left\vert r_{j}\right\vert ^{q_{j}}=r_{min}$ for all $j=1,..,k$. To see
this, assume the contrary, that there exists an IFS whose contractions
are multiplicatively independent. Then there exists an $i$ such that $r_i$ is
not multiplicatively dependent with $r_{min}$. Let $\nu_n$ be the word $%
i^{k_n}$ such that $|r_i|^{k_n} \leq r_{min}^n < |r_i|^{k_n-1}$. As $r_i$ is
multiplicatively independent of $r_{min}$ we see that for all $n \neq m $
that there does not exist an $f(x) = \pm r_{min}^{n-m}x + c$ where $f\circ
S_{\nu_n} = S_{\nu_m}$. Hence there must be an infinite number of
neighbourhood types, a contradiction. 
\end{remark}

Assume we are given probabilities $p_{j}>0$ satisfying $%
\sum_{j=0}^{k}p_{j}=1 $. There is a unique self-similar measure $\mu $
associated with the IFS $\{S_{j}\}_{j=0}^{k}$ and probabilities $%
\{p_{j}\}_{j=0}^{k}$ satisfying the rule 
\begin{equation*}
\mu =\sum_{j=0}^{k}p_{j}\mu \circ S_{j}^{-1}.
\end{equation*}%
This non-atomic probability measure $\mu $ has support the associated
self-similar set. Our interest is to study the dimensional properties of
self-similar measures associated with IFS of finite type.
Given $\sigma = (\sigma_1, \sigma_2, \dots, \sigma_j) \in \Sigma$, 
   we define \[ p_{\sigma }=\prod_{i=1}^{j}p_{\sigma _{i}}. \]

\subsection{Net intervals and characteristic vectors}

Next, we extend the definitions of net intervals and characteristic vectors
from equicontractive finite type IFS, to the non-equicontractive case.

\begin{definition}
For each positive integer $n$, let $h_{1},\dots ,h_{s_{n}}$ be the
collection of elements of the set $\{S_{\sigma }(0),$ $S_{\sigma }(1):\sigma
\in \Lambda _{n}\}$, listed in increasing order. Put 
\begin{equation*}
\mathcal{F}_{n}=\{[h_{j},h_{j+1}]:1\leq j\leq s_{n}-1\text{ and }%
(h_{j},h_{j+1})\cap K\neq \emptyset \}\text{.}
\end{equation*}%
Elements of $\mathcal{F}_{n}$ are called \textbf{net intervals of generation 
}$n$. The interval $[0,1]$ is understood to be the only net interval of
generation $0$.
\end{definition}

For each $\Delta \in \mathcal{F}_{n}$, $n\geq 1$, there is a unique element $%
\widehat{\Delta }\in \mathcal{F}_{n-1}$ which contains $\Delta ,$ called the 
\textit{parent} (of \textit{child} $\Delta )$. Given $\Delta =[a,b]\in 
\mathcal{F}_{n}$, we denote the \textit{normalized length} of $\Delta $ by 
\begin{equation*}
\ell _{n}(\Delta )=r_{\min }^{-n}(b-a)\text{.}
\end{equation*}%
By the \textit{neighbour set} of $\Delta $ we mean the ordered tuple 
\begin{equation*}
V_{n}(\Delta )=((a_{1},L_{1}),(a_{2},L_{2}),\dots ,(a_{j},L_{j})),
\end{equation*}%
where for each $i$ there is some $\sigma \in \Lambda _{n}$ such that 
$r_{\min }^{-n}r_{\sigma }=L_{i}$ and $r_{\min }^{-n}(a-S_{\sigma}(0))=a_{i}$.
That is, $L_{i}$ is the normalized length of the basic interval $S_{\sigma }[0,1]$. 
It is worth noting that $r_\sigma$, and hence $L_i$, may be negative.  Equivalently, there exists $\sigma \in \Lambda _{n}$ such that%
\begin{equation*}
S_{\sigma }(x)=r_{\min }^{n}(L_{i}x-a_{i})+a.
\end{equation*}%
We will order these tuples so that $a_{i}\leq a_{i+1}$ and if $a_{i}=a_{i+1},$ then 
$L_{i} < L_{i+1}$. (The precise order is not essential, although some consistent ordering is required.) Abusing terminology slightly, we will refer to one of
these $a_{i}$ (or a pair $(a_{i},L_{i}))$ as a \textit{neighbour} of $\Delta
. $

Suppose  $\Delta \in \mathcal{F}_n$ has parent $\widehat \Delta$.
It is possible for $\widehat \Delta$ to have multiple children with the 
    same normalized length and the same neighbourhood set as $\Delta$.
Order these equivalent children from left to right as $\Delta_1, \Delta_2, 
\dots, \Delta_t$.
We denote by $t_n(\Delta)$ the integer $t$ such that $\Delta_t = \Delta$.

\begin{definition}
The \textbf{characteristic vector}\textit{\ of }$\Delta \in \mathcal{F}_{n}$
is defined to be the triple 
\begin{equation*}
\mathcal{C}_{n}(\Delta )=(\ell _{n}(\Delta ),V_{n}(\Delta ),t_{n}(\Delta )).
\end{equation*}%
The characteristic vector of $[0,1]$ is defined as $(1,(0,1),1)$.
\end{definition}

Often we suppress $t_{n}(\Delta )$ giving the \textit{reduced characteristic
vector} $(\ell _{n}(\Delta ),V_{n}(\Delta ))$.

By the\textit{\ symbolic representation} of a net interval $\Delta \in 
\mathcal{F}_{n}$ we mean the $n+1$ tuple $(\mathcal{C}_{0}(\Delta _{0}),...,%
\mathcal{C}_{n}(\Delta _{n}))$ where $\Delta _{0}=[0,1]$, $\Delta
_{n}=\Delta $, and for each $j=1,...,n$, $\Delta _{j-1}$ is the parent of $%
\Delta _{j}$. Similarly, for each $x\in \lbrack 0,1]$, the symbolic
representation of $x$ will be the sequence of characteristic vectors%
\begin{equation*}
\lbrack x]=(\mathcal{C}_{0}(\Delta _{0}),\mathcal{C}_{1}(\Delta _{n}),...)
\end{equation*}%
where $x\in \Delta _{n}\in \mathcal{F}_{n}$ for each $n$ and $\Delta _{j-1}$
is the parent of $\Delta _{j}$. The symbolic representation uniquely
determines $x$ and is unique unless $x$ is the endpoint of some net
interval, in which case there can be two different symbolic representations.
We will write $\Delta _{n}(x)$ for a net interval of generation $n$
containing $x$.

Here are some important properties of characteristic vectors.

\begin{lemma}
\label{lem:children}
The characteristic vector of any child of the net interval $\Delta $ is
determined by the normalized length and neighbour set of $\Delta $.
\end{lemma}

\begin{proof}
Let $\Delta =[a,b]=[a,a+r_{\min }^{n}\ell _{n}(\Delta )]\in \mathcal{F}_{n}$
and assume that $V_{n}(\Delta )=((a_{i},L_{i}))_{i=1}^j$. The endpoints
of the children of $\Delta $ are those points $S_{\omega }(0),$ $S_{\omega
}(1)\in \lbrack a,b]$ for some $\omega \in \Lambda _{n+1}$. For each such $%
\omega ,$ there must be some $\sigma \in \Lambda _{n}$ and finite word $\tau 
$ such that $S_{\sigma }[0,1]$ $\supseteq \Delta $ and $\omega =\sigma \tau $%
. Such a $\sigma $ must satisfy the relation 
\begin{equation}
S_{\sigma }(x)=r_{\min }^{n}(L_{i}x-a_{i})+a  \label{Ssigma}
\end{equation}%
for some neighbour $(a_{i},L_{i})$ of $\Delta $. We claim the choices of
such $\tau $ can only be made from a finite set, say $F,$ determined by the $%
L_{i}$. This follows because if $\omega \in \Lambda _{n+1}$, we must have $%
\left\vert r_{\sigma \tau }\right\vert =r_{\min }^{n}\left\vert
L_{i}\right\vert \left\vert r_{\tau }\right\vert \leq r_{\min }^{n+1}$ and $%
\left\vert r_{(\sigma \tau )^{-}}\right\vert =\left\vert r_{\sigma \tau
^{-}}\right\vert >r_{\min }^{n+1}$. Therefore $\left\vert r_{\tau
}\right\vert \leq r_{\min }/\left\vert L_{i} \right\vert <  \left\vert r_{\tau ^{-}}\right\vert$ and there are only finitely many such $\tau $.

Using equation \eqref{Ssigma}, it follows that 
\begin{equation}
S_{\omega }(x)=a+r_{\min }^{n}(L_{i}S_{\tau }(x)-a_{i})=a+r_{\min
}^{n}(L_{i}r_{\tau }x+L_{i}S_{\tau }(0)-a_{i})  \label{SO}
\end{equation}%
for the appropriate $\tau \in F$. Thus the endpoints of the children of $%
\Delta $ are those of the form $a+r_{\min }^{n}c_{j}\in \lbrack a,a+r_{\min
}^{n}\ell _{n}(\Delta )]$ where $c_{j}=L_{i}S_{\tau }(0)-a_{i}$ or $%
c_{j}=L_{i}r_{\tau }+L_{i}S_{\tau }(0)-a_{i}$ for some $\tau \in F$. But
that is equivalent to saying $c_{j}\in \lbrack 0,\ell _{n}(\Delta )]$ for
some $c_{j}$ as specified, and hence the choice depends only upon $\ell
_{n}(\Delta ),$ $a_{i},L_{i}$ and the finite set $F,$ which depends only on
the $L_{i}$. Furthermore, the normalized lengths of these net subintervals
clearly also depend only on the $a_{i}$, $L_{i}$ and $\ell _n (\Delta )$.

Fix some child $[a+c_{j}r_{\min }^{n+1},a+c_{l}r_{\min }^{n+1}]$ of $\Delta $%
. Similar reasoning to the above shows that if this interval is covered by
some $S_{\omega }[0,1]$ for $\omega \in \Lambda _{n+1},$ then $S_{\omega }$
is as in equation \eqref{SO} for some $\tau \in F$ and the choice of $\tau $ is
independent of $a$. The normalized length of $S_{\omega }[0,1]$ and the
terms $r_{\min }^{-(n+1)}(S_{\omega }(0)-(a+c_{j}r_{\min }^{n+1}))$, and
thus the neighbour set of the child, also depend only on the $a_{i}$ and $L_{i}$.

This completes the proof that the characteristic vectors of the children
depend only upon the normalized length $\ell _{n}$ and neighbour set $%
((a_{i},L_{i}))$ of $\Delta $.
\end{proof}

\begin{theorem}
An IFS of finite type has only finitely many (distinct) characteristic
vectors.
\end{theorem}

\begin{proof}
Consider any net interval $\Delta \in \mathcal{F}_{n}$ and an $n\,$'th
generation basic interval, $S_{\sigma }[0,1],$ covering $\Delta $. Let $%
\mathcal{N(\sigma )}$ be the neighbours of $\sigma $. Suppose $\{e_{i}\}$ is
the (finite) set of all endpoints of $S_{\tau }[0,1]$ for $\tau \in \mathcal{%
N(\sigma )}$. The interval $\Delta $ must be of the form $[e_{s},e_{t}]$ for
a suitable choice (of adjacent endpoints). Moreover, any $n$'th generation
basic interval covering $\Delta $ arises from some $\tau \in \mathcal{%
N(\sigma )}$. It follows that the normalized length of $\Delta $, the
normalized length of $S_{\sigma }[0,1]$ and $r_{\min }^{-n}(e_{i}-$ $S_{\tau
}(0))$ for any $\tau \in \Lambda _{n}$ with $S_{\tau }[0,1]\supseteq \Delta $
are all in the finite set of normalized differences $\{r_{\min }^{-n}(e_{i}-$
$e_{j})\}$. This set of normalized differences is the same for all $\sigma
^{\prime }\in \Lambda _{n}$ which have the same neighbourhood type as $%
\sigma $. As there are only finitely many neighbourhood types, there can
only be finitely many normalized differences. But the elements of a
characteristic vector are chosen from these, so there are only finitely many
characteristic vectors.
\end{proof}

We will denote the set of characteristic vectors by 
\begin{equation*}
\Omega =\{\mathcal{C}_{n}(\Delta ):n\in \mathbb{N}\text{, }\Delta \in 
\mathcal{F}_{n}\}\text{.}
\end{equation*}

\section{Local Dimensions and Transition Matrices}
\label{sec:local}

\subsection{Local dimensions of measures of finite type}

\begin{definition}
Given a probability measure $\mu $, by the \textbf{upper local dimension} of 
$\mu $\textit{\ }at $x\in $ supp$\mu $, we mean the number 
\begin{equation*}
\overline{\dim}_{loc}\mu (x)=\limsup_{r\rightarrow 0^{+}}\frac{\log \mu
([x-r,x+r])}{\log r}.
\end{equation*}%
Replacing the $\limsup $ by $\liminf $ gives the \textbf{lower local
dimension}, denoted $\underline{\dim}_{loc}\mu (x)$. If the limit exists, we
call the number the \textbf{local dimension} of $\mu $ at $x$ and denote
this by $\dim_{loc}\mu (x)$.
\end{definition}

It is easy to see that when the limit exists 
\begin{equation}
\dim _{loc}\mu (x)=\lim_{n\rightarrow \infty }\frac{\log \mu
([x-r^{n},x+r^{n}])}{n\log r}\text{ for }x\in \text{supp}\mu ,
\label{locdim}
\end{equation}%
for any fixed $0<r<1$. Similar statements hold for upper and lower local
dimensions. As all net intervals of generation $n$ have lengths comparable to 
$r_{\min }^{n}$, it follows that for a measure $\mu $ of finite type%
\begin{equation}
\dim _{loc}\mu (x)=\lim_{n\rightarrow \infty }\frac{\log (\mu (\Delta
_{n}(x))+\mu (\Delta _{n}^{+}(x))+\mu (\Delta _{n}^{-}(x)))}{n\log r_{\min }},
\label{locdimft}
\end{equation}%
where $\Delta _{n}(x)$ is the net interval of generation $n$ containing $x$
and $\Delta _{n}^{+}(x),\Delta _{n}^{-}(x)$ are the adjacent, $n$'th
generation net intervals on each side (with the understanding that in the
case that $x$ is a boundary point, we can replace these three net intervals
by the two net intervals having $x$ as a common boundary point.)

Thus we are interested in calculating $\mu (\Delta )$ for net intervals $%
\Delta $. We begin with some technical results.

\begin{lemma}
Suppose $\Delta =[a,b]\in \mathcal{F}_{n}$ has normalized length $\ell
_{n}(\Delta )$ and neighbour set $V_{n}(\Delta )=((a_{i},L_{i}))_{i=1}^{j}$.
For each $i=1,...,j$, let $F_{i}(x)=L_{i}^{-1}(\ell _{n}(\Delta )x+a_{i})$.
Then 
\begin{equation}
\mu (\Delta )=\sum_{i=1}^{j}\mu (F_{i}[0,1])\sum_{\substack{ \sigma \in
\Lambda _{n}  \\ S_{\sigma }(x)=r_{\min }^{n}(L_{i}x-a_{i})+a}}p_{\sigma }.
\label{firstcalc}
\end{equation}
\end{lemma}

\begin{proof}
This follows from similar arguments to \cite{F3}, which we will sketch here.
From the defining identity for a self-similar measure, we have that 
\begin{equation}
\mu (\Delta )=\sum_{\sigma \in \Lambda _{n}}p_{\sigma }\mu (S_{\sigma
}^{-1}(\Delta )).
\label{eq:define}
\end{equation}%
In fact, as $\mu $ is non-atomic, the sum can be taken over those $\sigma $
such that $S_{\sigma }^{-1}(\Delta )\bigcap (0,1)$ is non-empty.  In this
case, the definition of a net interval ensures $\Delta \subseteq S_{\sigma
}[0,1].$ Hence $S_{\sigma }(x)=r_{\min }^{n}(L_{i}x-a_{i})+a$ for some $i\in
\{1,...,j\}$. The maps $F_{i}$ were defined so that $F_{i}[0,1]=S_{\sigma
}^{-1}(\Delta )$, hence the conclusion of the lemma.
\end{proof}

\begin{notation}
Given $\Delta \in \mathcal{F}_{n},$ $n\geq 1,$ with neighbour set $%
V_{n}(\Delta )=((a_{i},L_{i}))_{i=1}^{j}$, let 
\begin{equation*}
P_{n}^{i}(\Delta )=\sum_{\substack{ \sigma \in \Lambda _{n}  \\ S_{\sigma
}(x)=r_{\min }^{n}(L_{i}x-a_{i})+a}}p_{\sigma }\text{ for }i=1,...,j
\end{equation*}%
and let 
\begin{eqnarray*}
Q_{n}(\Delta ) &=&(P_{n}^{1}(\Delta ),...,P_{n}^{j}(\Delta ))\text{, }%
Q_{0}[0,1]=(1) \\
\text{ and }P_{n}(\Delta ) &=&\sum_{i=1}^{j}P_{n}^{i}(\Delta )=\left\Vert
Q_{n}(\Delta )\right\Vert.
\end{eqnarray*}
where here $\Vert \cdot \Vert$ denotes the $1$-norm.
\end{notation}

\begin{corollary}
There is a constant $c>0$ such that for any $n$ and $\Delta \in \mathcal{F}%
_{n}$, we have 
\begin{equation}
cP_{n}(\Delta )\leq \mu (\Delta )\leq P_{n}(\Delta )\text{.}  \label{PTcomp}
\end{equation}
\end{corollary}

\begin{proof}
Note that $\mu(S_{\sigma}^{-1}(\Delta)) > 0$ for all $\sigma$ appearing in the 
    summand of equation \eqref{eq:define}.
For each $\sigma$ in this summand, there is an $i$ such that 
    $S_\sigma^{-1}(\Delta) = F_i([0,1])$.
There are only finitely many $F_i$, as there are only finitely many characteristic vectors.
From this, the result follows.
\end{proof}

\subsection{Transition matrices}

We can calculate $P_{n}$ by means of transition matrices which take us from
generation $n-1$ to generation $n$.

\begin{definition}
Let $\Delta =[a,b]\in \mathcal{F}_{n}$ and let $\widehat{\Delta }=[c,d]\in
\mathcal{F}_{n-1}$ denote its parent net interval. Assume
\begin{eqnarray*}
V_{n}(\Delta ) &=&((a_{1},L_{1}),\dots ,(a_{I},L_{I}))\text{ and } \\
V_{n-1}(\widehat{\Delta }) &=&((c_{1},M_{1}),\dots ,(c_{J},M_{J})).
\end{eqnarray*}%
The \textbf{primitive transition matrix}, $T(%
\mathcal{C}_{n-1}(\widehat{\Delta }),\mathcal{C}_{n}(\Delta ))$, is the $%
I\times J$ matrix $(T_{ij})$ which encapuslates information about the
relationship between the $(c_{i},M_{i})\in V_{n-1}(\widehat{\Delta })$ and $%
(a_{j},L_{j})\in V_{n}(\Delta )$. To be precise, let $\sigma _{i}\in \Lambda
_{n-1}$ be such that $\rho _{\min }^{-n+1}(c-S_{\sigma }(0))=c_{i}$ and $%
\rho _{\min }^{-n+1}\rho _{\sigma }=M_{i}$. Let $\mathcal{T}_{i,j}$ be the
set all $\omega $ such that $\sigma \omega \in \Lambda _{n}$, $\rho _{\min
}^{-n}(a-S_{\sigma _{i}\omega }(0))=a_{j}$ and $\rho _{\min }^{-n}\rho
_{\sigma _{i}\omega }=L_{j}$. Notice that $\mathcal{T}_{i,j}$ depends only $%
S_{\sigma _{i}}$ (or equivalently on $c_{i}$ and $M_{i}$), and not on the
choice of $\sigma _{i}$. We define $T_{i,j}=\sum_{\omega \in \mathcal{T}%
_{i,j}}p_{\omega }$ where the empty sum is taken to be $0$.
\end{definition}

Every column of a primitive transition matrix has at least one non-zero
entry since the existence of $(c_{j},M_{j})$ means there is some $\tau \in
\Lambda _{n}$ with $S_{\tau }[0,1]\supseteq \Delta $ and $S_{\tau
}(x)=r_{\min }^{n}(M_{j}x-c_{j})+c$. Taking as $\sigma $ the initial segment
of $\tau $ of generation $n-1$, we see that $S_{\sigma }[0,1]$ is a basic
interval of generation $n-1$ containing $\Delta ,$ and therefore must also
contain $\widehat{\Delta }$. Consequently there is some $(a_{i},$ $L_{i})$
such that $S_{\sigma }(x)=r_{\min }^{n-1}(L_{i}x-a_{i})+a$ and that means $%
T_{ij}\neq 0$.

When the support of $\mu $ is the full interval $[0,1],$ then it also
follows that each row has a non-zero entry.

The same reasoning as in \cite{F3} shows that 
\begin{equation*}
Q_{n}(\Delta )=Q_{n-1}(\widehat{\Delta })T(\mathcal{C}_{n-1}(\widehat{\Delta 
}),\mathcal{C}_{n}(\Delta )).
\end{equation*}%
More generally, if $\Delta $ has symbolic representation $(\gamma
_{0},\gamma _{1},...,\gamma _{n})$, then 
\begin{equation*}
P_{n}(\Delta )=\left\Vert T(\gamma _{0},\gamma _{1})\cdot \cdot \cdot
T(\gamma _{n-1},\gamma _{n})\right\Vert
\end{equation*}%
where the matrix norm is given by $\left\Vert (M_{ij})\right\Vert
=\sum_{ij}\left\vert M_{ij}\right\vert$.  We will write $T(\gamma
_{0},\gamma _{1},...,\gamma _{n})$ for $T(\gamma _{0},\gamma _{1})\cdot
\cdot \cdot T(\gamma _{n-1},\gamma _{n})$ and call this a \textit{transition
matrix}.

We note that if $A$ and $B$ are transition matrices, then since $A$ has a non-zero
entry in each column an easy exercise shows $\left\Vert AB\right\Vert \geq
c\left\Vert B\right\Vert ,$ where $c>0$ depends only on $A$. Thus if $%
[x]=(\gamma _{0},\gamma _{1},...)$ and $\Delta _{n}$ is the net interval
containing $x$ with symbolic representation $(\gamma _{0},\gamma
_{1},...,\gamma _{n})$, then for any $J \le n$.
\begin{equation*}
\mu (\Delta _{n})\sim P_{n}(\Delta )\sim \left\Vert T(\gamma _{J},\gamma
_{J+1},...,\gamma _{n})\right\Vert , 
\end{equation*}%
where the constants of comparability depend only on $\gamma _{0},...,\gamma
_{J-1}$. Here when we write $A_n \sim B_n$ we mean there are positive constants $c_1, c_2$ such that $$
A_n \le c_1 B_n \le c_2 A_n.$$

\subsection{Local dimensions at periodic points\label{LocPeriodic}} \label{PerPt}

A \textit{periodic point} $x$ is a point with symbolic representation%
\begin{equation*}
\lbrack x]=(\gamma _{1},...,\gamma _{J},\theta ^{-},\theta ^{-},...),
\end{equation*}%
where $\theta =(\theta _{1},...,\theta _{s},\theta _{1})$ is a cycle
(meaning, the first and last letters are the same) and $\theta ^{-}$ has the
last letter of $\theta $ deleted. We call $\theta $ a \textit{period} of $x$%
. Of course, periods are not unique; $(\theta ^{-},\theta )$ and $(\theta
_{2},...,\theta _{s},\theta _{1},\theta _{2})$ are other periods, for
instance.

An example of a periodic point is a boundary point. There are only finitely many periods of minimal length associated with boundary points.  We will denote by $\lambda $ the least
common multiple of these minimum period lengths. For any boundary point, there is always a choice of
period whose length is $\lambda $ and we will normally assume such a choice
has been made.

Let $x$ be a boundary point, and consider a symbolic representation of $x$.
We see that sufficiently long initial segments of this symbolic 
representation will represent a net interval with $x$ as one of its endpoints.
Typically, boundary points also have a
second representation where the corresponding net intervals have $x$ as
their opposite endpoint. We can write these two symbolic representations as 
\begin{equation*}
(\gamma _{1},...,\gamma _{J},\psi ,\theta ^{-},\theta ^{-},...)\text{ and }%
(\gamma _{1},...,\gamma _{J},\psi ^{\prime },\theta ^{\prime -},\theta
^{\prime -},...)
\end{equation*}
where $(\gamma _{1},...,\gamma _{J})$ is the symbolic representation of the
highest generation net interval containing $x$ as an interior point. We
refer to $\psi ,\psi ^{\prime }$ as preperiods. Again, these are not unique,
but there is a unique choice of minimal length and this minimal length is
bounded over all preperiods associated with the IFS. We will let this bound
be denoted by $\kappa $ and observe that by extending $\psi ,$ if necessary
(by adding in some letters from $\theta$ and cycling $\theta$ as necessary),
we can assume the preperiod of boundary points has length $\kappa $.

With the assumptions that the preperiod length is $\kappa $ and the period
length is $\lambda ,$ this gives two uniquely defined symbolic
representations for a typical boundary point. We will refer to these as the 
\textit{standard symbolic representations}.

Here is the formula for the local dimensions at periodic points. Note that $%
T(\theta )$ is a square matrix as $\theta $ is a cycle. We denote by $%
sp(T(\theta ))$ its spectral radius.

\begin{proposition}
\label{periodic}If $x$ is a periodic point with period $\theta $ of period
length $\beta $, then the local dimension of $\mu $ at $x$ exists and is
given by 
\begin{equation*}
\dim _{loc}\mu (x)=\frac{\log sp(T(\theta ))}{\beta \log r_{\min }},
\end{equation*}%
where if $x$ is a boundary point of a net interval with two different
standard symbolic representations given by periods $\theta $ and $\theta
^{\prime }$ of length $\beta = \lambda ,$ then $\theta $ is chosen to satisfy $%
sp(T(\theta ))\geq sp(T(\theta ^{\prime }))$.
\end{proposition}

\begin{proof}
The proof is the same as that found in \cite[Proposition 2.7]{HHN} for
equicontractive, finite type self-similar measures. The key idea is that if $[x]=$ $(\gamma
_{1},...,\gamma _{J},\psi ,\theta ^{-},\theta ^{-},...)$ and $\Delta
_{n}(x)=(\gamma _{1},...,\gamma _{J},\psi ,\underbrace{\theta
^{-},...,\theta }_{m})$ is the associated $n$'th generation net interval for 
$n=J+\kappa +m\beta $, then $\mu (\Delta _{n}(x))\sim \left\Vert \left(
T(\theta )\right) ^{m}\right\Vert $. Thus%
\begin{eqnarray*}
\lim_{n\rightarrow \infty }\frac{\log \mu (\Delta _{n}(x))}{n\log r_{\min }}
&=&\lim_{n\rightarrow \infty }\frac{\log \left\Vert T(\theta
)^{m}\right\Vert }{n\log r_{\min }}=\lim_{m\rightarrow \infty }\frac{\log
\left\Vert T(\theta )^{m}\right\Vert ^{1/m}}{\frac{n}{m}\log r_{\min }} \\
&=&\frac{\log sp(T(\theta ))}{\beta \log r_{\min }}.
\end{eqnarray*}

If $x$ is a boundary point with the second standard symbolic representation $(\gamma
_{1},...,\gamma _{J},\psi ^{\prime },\theta ^{\prime -},\theta ^{\prime
-},...),$ then $\Delta _{n}^{\prime }(x)=(\gamma _{1},...,\gamma _{J},\psi
^{\prime },\underbrace{\theta ^{\prime -},...,\theta ^{\prime }}_{m})$ is
the net interval adjacent to $\Delta _{n}(x)$, with $x$ as the common boundary
point. We have%
\begin{equation*}
\lim_{n\rightarrow \infty }\frac{\log \mu (\Delta _{n}^{\prime }(x))}{n\log
r_{\min }}=\frac{\log sp(T(\theta ^{\prime }))}{\lambda \log r_{\min }}.
\end{equation*}%
Using these facts, we can readily deduce that 
\begin{eqnarray*}
\dim _{loc}\mu (x) &=&\lim_{n\rightarrow \infty }\frac{\log (\mu (\Delta
_{n}(x))+\mu (\Delta _{n}^{\prime }(x)))}{n\log r_{\min }} \\
&=&\min \left( \frac{\log sp(T(\theta ))}{\lambda \log r_{\min }},\frac{\log
sp(T(\theta ^{\prime }))}{\lambda \log r_{\min }}\right) .
\end{eqnarray*}

When $x$ is a periodic point that is not a boundary point, the ideas are
slightly more technical, but similar.
\end{proof}

\section{Comparable Self-similar Measures} 
\label{sec:comparable}

\subsection{Definitions of comparable and generalized regular} \label{CompReg}

For points $x\in K$ that are not periodic it can be complicated to compute
the local dimension as one needs to consider the $\mu $-measure of not only
the net intervals containing $x$ of each generation, but also the two
adjacent net intervals. We can avoid this complication when the self-similar
measure is what we call comparable.

\begin{definition}
We will call a self-similar measure associated with a finite type IFS 
\textbf{comparable} if for each $q>1$, there is a constant $c=c(q)>0$ such
that for all positive integers $n$ and adjacent net intervals $\Delta
_{1},\Delta _{2}$ of generation $n$, 
\begin{equation}
\frac{1}{c}q^{-n}P_{n}(\Delta _{2})\leq P_{n}(\Delta _{1})\leq
cq^{n}P_{n}(\Delta _{2}).  \label{comp}
\end{equation}
\end{definition}

\begin{example}
If $\mu $ is a self-similar measure associated with an equicontractive IFS
where the probabilities satisfy $p_{0}=p_{k}=\min p_{i}$ (where $S_{0}(0)=0$
and $S_{k}(1)=1),$ then $\mu $ is comparable and even satisfies the stronger
inequality%
\begin{equation*}
\frac{1}{cn}P_{n}(\Delta _{2})\leq P_{n}(\Delta _{1})\leq cnP_{n}(\Delta
_{2}).
\end{equation*}%
See \cite[Corollary 2.12]{F2} or \cite[Lemma 3.5]{HHM}. These equicontractive
measures with $p_0 = p_k = \min p_i$ are referred to as \textbf{regular}.
\end{example}

It is easy to give an elegant formula for the local dimension of a
comparable self-similar measure.

\begin{theorem}
\label{CompFormula}Assume the self-similar measure $\mu $ arises from an IFS
that is of finite type. If $\mu $ is comparable, then for any $x$ with
symbolic representation $(\gamma _{0},\gamma _{1},\gamma _{2},...),$%
\begin{align*}
\dim _{loc}\mu (x) & =\lim_{n}\frac{\log \mu(\Delta _{n}(x))}{n\log r_{\min }}\\
 & =\lim_{n}\frac{\log P_{n}(\Delta _{n}(x))}{n\log r_{\min }} \\
& =\lim_{n}\frac{\log \left\Vert T(\gamma _{0},\gamma _{1},\gamma
_{2},...\gamma _{n})\right\Vert }{n\log r_{\min }},  
\end{align*}
(should the limit exist).  Similar formulas hold for the lower and
upper local dimensions.
\end{theorem}

\begin{proof}
Given $x$, consider $\Delta _{n}(x)$ and the two adjacent net intervals $%
\Delta _{n}^{+}(x)$ and $\Delta _{n}^{-}(x)$ of generation $n$. As $\mu $ is
comparable, equation \eqref{comp} holds with $\Delta _{1}=\Delta _{n}(x)$ and $\Delta
_{2}$ either $\Delta _{n}^{+}(x)$ or $\Delta _{n}^{-}(x)$. We also know from
equation \eqref{PTcomp} that there exists $c^{\prime }>0$ such that $c^{\prime
}P_{n}(\Delta )\leq \mu (\Delta )\leq P_{n}(\Delta )$ for all net intervals $%
\Delta $ of generation $n$.

Putting these facts togther, we obtain%
\begin{equation*}
c^{\prime }P_{n}(\Delta _{n}(x))\leq \mu (\Delta _{n}(x))+\mu (\Delta
_{n}^{+}(x))+\mu (\Delta _{n}^{-}(x))\leq 3cq^{n}P_{n}(\Delta _{n}(x)).
\end{equation*}%
Thus%
\begin{eqnarray*}
\frac{\log 3c}{n\log r_{\min }}+\frac{n\log q}{n\log r_{\min }}+\frac{\log
P_{n}(\Delta _{n}(x))}{n\log r_{\min }} &\leq &\frac{\log (\mu (\Delta
_{n}(x))+\mu (\Delta _{n}^{+}(x))+\mu (\Delta _{n}^{-}(x)))}{n\log r_{\min }}
\\
&\leq &\frac{\log c^{\prime }}{n\log r_{\min }}+\frac{\log P_{n}(\Delta
_{n}(x))}{n\log r_{\min }}.
\end{eqnarray*}%
The conclusion of the theorem follows by letting $n\rightarrow \infty $ and
noting that $q>1$ is arbitrary.
\end{proof}

To define an analogue of the notion of `regular' (which we call `generalized regular') in the non-equicontractive
case, it is helpful to introduce some terminology. By a \textit{path of
generation }$n$ \textit{(associated with }$\sigma ),$ we mean a word $\omega
\in \Sigma ^{\ast }$ such that there is some $m$ and $\sigma \in \Lambda
_{m} $ such that $\sigma \omega \in \Lambda _{n+m}$. Of course, in the
equicontractive case, a path of generation $n$ is a word of length $n$.

By a \textit{left-edge path of generation }$n,$ we mean a path of generation 
$n$ associated with some $\sigma $ such that $S_{\sigma }[0,1]$ and $%
S_{\sigma \omega }[0,1]$ have the same left endpoint; equivalently, 
\begin{equation*}
\min \left( S_{\sigma }(0),S_{\sigma }(1)\right) =\min \left( S_{\sigma
\omega }(0),S_{\sigma \omega }(1)\right) .
\end{equation*}%
In this case, there will be a net interval $\Delta \in \mathcal{F}_{m}$,
    where $\Delta \subset S_\sigma[0,1]$, and further, $\Delta$ and $S_\sigma[0,1]$ share
    the same left endpoint.
Moreover, the descendent $\Delta ^{\prime
}\in \mathcal{F}_{m+n}$, with the same left endpoint will be contained in $%
S_{\sigma \omega }[0,1]$ and this latter interval will also share the same
left endpoint. If we want to emphasize the connection with $\Delta $, we
will say that $\omega $ is a \textit{left-edge path of }$\Delta $. We define
right-edge paths similarly and an edge path is either a right or left-edge
path.

\begin{example}
If $0=S_{0}(0)$, $0\notin S_{j}[0,1]$ for any $j=1,...,k$, and all $r_{i}>0$%
, then the only left-edge paths are of the form $(0)^{j}$ for some $j$.
\end{example}

\begin{notation}\label{regnotation}
Given a net interval $\Delta $ and $n\in \mathbb{N}$, we let%
\begin{equation*}
\Gamma _{\Delta ,n}=\sum_{\substack{ \omega \text{ edge path of }\Delta  \\ 
\text{of generation }n}}p_{\omega }.
\end{equation*}%
Let $R_{n}$ denote the minimum transition ratio between net intervals $n$
generations apart, meaning:%
\begin{equation*}
R_{n}=\inf_{\substack{ m,\Delta \in \mathcal{F}_{m},  \\ \Delta ^{\prime
}\subseteq \Delta ,\Delta ^{\prime }\in \mathcal{F}_{n+m}}}\frac{%
P_{m+n}(\Delta ^{\prime })}{P_{m}(\Delta )},
\end{equation*}%
i.e., the infimum is taken over all positive integers $m$ and net intervals $%
\Delta ^{\prime }\subseteq \Delta $ of generations $m+n$ and $m$
respectively. Finally, let%
\begin{equation*}
B(n)=\frac{1}{R_{n}}\sup_{\text{all }\Delta }\Gamma _{\Delta ,n}.
\end{equation*}
\end{notation}

We note that for any $\Delta \in \mathcal{F}_{m}$ and descendent $\Delta
^{\prime }\in \mathcal{F}_{m+n}$, we have 
\begin{equation*}
P_{m+n}(\Delta ^{\prime })\geq \inf \{p_{\omega }:\omega \text{ path of
generation }n\}P_{m}(\Delta ).
\end{equation*}%
Hence $R_{n}\geq \inf p_{\omega }$ where the infimum is taken over all paths
of generation $n$. Of course, that implies $R_{n}\geq \min_{i}p_{i}^{n} >0$.

\begin{definition}\label{regdefn}
\label{regdef}We will say that a self-similar measure associated with a
finite type IFS is \textbf{generalized regular} if for each $q>1$, $\lim_{n}B(n)/q^{n}=0$%
.
\end{definition}

\begin{remark} \label{RegSub}
Notice that for any fixed $i$, $\Gamma _{\Delta ,n}\sim \Gamma _{\Delta
,n+i} $ and $R_{n}\sim R_{n+i}$, with the constants of comparability
depending only on $i$. This is because there are only finitely many paths of
generation $d$. Thus in verifying generalized regularity it is enough to check that the
requirement of Definition \ref{regdef} holds along the subsequence $%
(B(dn))_{n}$, for example.
\end{remark}

Generalized regularity is a useful concept as it implies comparability, as we prove in
Theorem \ref{regimpliescomp}. But first, we give some conditions that ensure
generalized regularity.

\begin{example}
Assume the IFS is equicontractive with $S_{0}(0)=0,S_{k}(1)=1$ and $%
r_{i}=r>0 $ for all $i$. Then the edge paths of generation $n$ are the
length $n$ words $(0)^{n}$ and $(k)^{n},$ hence $\Gamma _{\Delta ,n}\leq
p_{0}^{n}+p_{k}^{n}$ for all $\Delta $. As $R_{n}\geq \min_{i}p_{i}^{n}$, it
follows that if $p_{0}=p_{k}=\min p_{i}$, then the IFS\ is generalized regular. Thus our
notion of generalized regular extends the notion of regular from the equicontractive
case.
\end{example}

\begin{remark}
Example \ref{ExReg} shows that this definition of generalized regular includes self-similar measures arising from equicontractive IFS that are not regular in the sense of the original definition. 
\end{remark}

\begin{proposition}
\label{reg}Assume the finite type IFS $\{S_{j}:j=0,...,k\}$ has the property
that $0\in S_{0}[0,1]$, $1\in S_{k}[0,1]$ and $0,1\notin S_{j}[0,1]$ for $%
j\neq 0,k$. Let $\mu $ be an associated self-similar measure with
probabilities $\{p_{j}\}$.
\end{proposition}

\begin{enumerate}
\item If $r_{0},r_{k}>0$, then $\mu $ is generalized regular if 
\begin{equation*}
\frac{\log p_0}{\log r_0} = \frac{\log p_k}{\log r_k} \geq
\max_{j=1,...,k-1}\frac{\log p_{j}}{\log \left\vert r_{j}\right\vert }.
\end{equation*}

\item If $r_{0}<0$ and $r_{k}>0$, then $\mu $ is generalized regular if 
\begin{equation*}
\frac{\log p_{k}}{\log r_{k}}\geq \max_{j=0,...,k}\frac{\log p_{j}}{\log
\left\vert r_{j}\right\vert }.
\end{equation*}%
A similar statement holds if $r_{0}>0$ and $r_{k}<0$.

\item If $r_{0},r_{k}<0$, then $\mu $ is generalized regular if 
\begin{equation*}
\frac{\log p_0}{\log\left\vert r_0\right\vert} = \frac{\log p_k}{\log\left\vert r_k\right\vert} \geq
\max_{j=1,...,k-1}\frac{\log p_{j}}{\log \left\vert r_{j}\right\vert }.
\end{equation*}
\end{enumerate}

\begin{proof}
(1) In this case, $\Gamma_{\Delta, n} \le c(p_0^{N_n} + p_k^{M_n})$ for suitable choices of $N_{n},M_{n}$ and $c$ independent of $\Delta$. Thus
it is enough to prove that there is some constant $C$ such that for all $n$
and for every path $\omega = (\omega_1,...,\omega_\ell) $ of generation $n$, we have 
\begin{equation*}
Cp_{0}^{N_{n}},Cp_{k}^{M_{n}}\leq p_{\omega }.
\end{equation*}%
Indeed, as noted in Remark \ref{RegSub} it is enough to prove this for the integers $n$ that are multiplies
of a fixed constant $d,$ and that is what we will actually show.

Using the elementary fact that $\alpha _{i},\beta _{i}<0$ and $z\geq \min (%
\frac{\alpha _{1}}{\beta _{1}},\frac{\alpha _{2}}{\beta _{2}})$ implies $%
z\geq (\alpha _{1}+\alpha _{2})/(\beta _{1}+\beta _{2})$, we can deduce from
the hypothesis that%
\begin{equation*}
\frac{\log p_{0}}{\log r_{0}}\geq \frac{\Sigma _{i=1}^{\ell}\log p_{\omega _{i}}%
}{\Sigma _{i}\log \left\vert r_{\omega _{i}}\right\vert }=\frac{\log
p_{\omega }}{\log \left\vert r_{\omega }\right\vert }.
\end{equation*}%
Thus, for $\sigma =(0)^{N_{n}}$, a left-edge path of generation $n$, we have 
\begin{equation*}
\log p_{\sigma }\frac{\log \left\vert r_{\omega }\right\vert }{\log
r_{\sigma }}\leq \log p_{\omega }.
\end{equation*}%
An easy calculation shows that if $\omega $ is any path of generation $n$,
then 
\begin{equation*}
r_{\min }^{n+1}\leq \left\vert r_{\omega }\right\vert \leq r_{\min }^{n-1}.
\end{equation*}%
Thus $\log \left\vert r_{\omega }\right\vert /\log r_{\sigma }\leq
(n+1)/(n-1)$ and therefore 
\begin{equation*}
p_{\omega }\geq p_{\sigma }^{\frac{n+1}{n-1}}.
\end{equation*}

As the lengths of a finite type IFS are commensurate, (see Remark \ref{commensurate}) there must be integers 
$b,d$ such that $r_{\min }^{d}=r_{0}^{b}$. We will now check that $%
p_{0}^{N_{n}}\leq Cp_{\omega }$ for integers $n$ that are multiples of $d$,
say $n=dm$. For convenience, put $p=p_{0}^{b/d}$. As $r_{\min
}^{n}=r_{0}^{bm}$ and $\sigma =(0)^{N_{n}}$ is of generation $n$, it follows
that $N_{n}=bm$. Thus $p_{\sigma }=p_{0}^{N_{n}}=p_{0}^{bm}=p^{n}$. Hence
\begin{equation*}
p_{\omega }\geq p^{n\left( \frac{n+1}{n-1}\right) }\geq p^{4}p^{n}\equiv
Cp_{0}^{N_{n}}
\end{equation*}%
as we desired to show.

Similar arguments show $p_{\omega }\geq C^{\prime }p_{k}^{M_{n}}.$

(2) When $r_{0}<0$ and $r_{k}>0$, then the only edge paths are of the form $%
(k)^{n}$ and $(0,(k)^{n})$. Thus if $\sigma $ is an edge path, then 
$\log p_{\sigma }\sim \log p_{(k)^{n}}$ and $\log \left\vert r_{\sigma
}\right\vert \sim \log \left\vert r_{(k)^{n}}\right\vert ,$ so the arguments
are similar to the previous case.

(3) In this case, the edge paths are of the form $(k,(0,k)^{n},0)$ where the
first $k$ or last $0$ need not be present. Similar reasoning shows that the
measure is generalized regular provided%
\begin{equation*}
\frac{\log p_{0}p_{k}}{\log r_{0}r_{k}}\geq \max_{j=0,...,k}\frac{\log p_{j}%
}{\log \left\vert r_{j}\right\vert }\text{,}
\end{equation*}%
and this is equivalent to what is claimed in the statement of the
proposition.
\end{proof}

\subsection{Generalized regular implies comparable}

We continue to use the notation $B(n)$, $R_{n},\kappa $
and $\lambda $ introduced in Subsections \ref{PerPt} and \ref{CompReg}. 

\begin{theorem}
\label{regimpliescomp}If the self-similar measure $\mu $ is associated with
an IFS of finite type and is generalized regular, then $\mu $ is comparable.
\end{theorem}

\begin{proof}
Fix $q>1$. As $B(n)/q^{n}\rightarrow 0$, we may replace $\lambda $ by a suitably
large multiple, so that $B(\lambda )<q^{\lambda }/2$. We will proceed
by induction on the generation of the net intervals. 
Let $N = \kappa + \lambda$ $(\kappa $ and $\lambda $ as defined in
Subsection \ref{LocPeriodic}).
Note, for $n \leq N$ there exists a constant $c$ such that the inequalities%
\begin{equation}
\frac{1}{c}q^{-n}P_{n}(\Delta _{2})\leq P_{n}(\Delta _{1})\leq
cq^{n}P_{n}(\Delta _{2})  \label{ProveReg}
\end{equation}%
hold for all adjacent net intervals $\Delta _{1},\Delta _{2}$ of level $n$.
This is because there are only finitely many such net
intervals and $P_{n}(\Delta )$ is always positive. There is no loss of
generality in assuming $c>2/R_{\lambda }$.

Now assume $\Delta _{1}$ and $\Delta _{2}$ are adjacent net intervals of
level $n>N$. If they have a common ancestor, $\widehat{\Delta },$ at level $%
n-k\,$, $k\leq N$, then equation \eqref{ProveReg} holds (with possibly a different
constant $c$, but depending only on the finitely many choices for $k)$ since 
\begin{equation*}
R_{k}P_{n-k}(\widehat{\Delta })\leq P_{n}(\Delta _{1}),P_{n}(\Delta
_{2})\leq P_{n-k}(\widehat{\Delta }).
\end{equation*}

So suppose otherwise. Assume the common boundary point $x$ of $\Delta _{1}$
and $\Delta _{2}$ has the two standard symbolic representations, 
\begin{equation*}
\lbrack x]=(\gamma _{1},...,\gamma _{J},\psi ,\theta ^{-},\theta ^{-},...)%
\text{, }(\gamma _{1},...,\gamma _{J},\psi ^{\prime },\theta ^{\prime
-},\theta ^{\prime -},...),
\end{equation*}%
with the notation as before. As $\Delta _{1}$ and $\Delta _{2}$ do not have
a common ancestor within $\kappa +\lambda $ levels back, it must be that $\Delta
_{1}$ has symbolic representation 
\begin{equation*}
(\gamma _{1},...,\gamma _{J},\psi ,\underbrace{\theta ^{-},\theta
^{-},...,\theta ^{-}}_{m},\theta _{1},...,\theta _{s})
\end{equation*}
for some $m\geq 1,$ where $\theta _{1},...,\theta _{s}$ is an initial
segment of $\theta ,$ possibly empty. A similar statement holds for $\Delta
_{2}$. In particular, $n=J+\kappa +m\lambda +s.$

Since $P_{n}(\Delta _{1})\sim P_{n+\lambda -s}(\Delta _{1}^{\ast })$ where $%
\Delta _{1}^{\ast }=(\gamma _{1},...,\gamma _{J},\psi ,\underbrace{\theta
^{-},\theta ^{-},...,\theta ^{-}}_{m+1})$, there is no loss of generality in
assuming%
\begin{equation*}
\Delta _{1}=(\gamma _{1},...,\gamma _{J},\psi ,\underbrace{\theta
^{-},\theta ^{-},...,\theta ^{-}}_{m})
\end{equation*}%
for some $m\geq 2$ and similarly for $\Delta _{2}$. Let $\widehat{\Delta _{1}%
}$ and $\widehat{\Delta _{2}}$ be the ancestors $\lambda $ levels earlier;
in other words, $\widehat{\Delta _{1}}=(\gamma _{1},...,\gamma _{J},\psi ,%
\underbrace{\theta ^{-},\theta ^{-},...,\theta ^{-}}_{m-1})$ and similarly
for $\widehat{\Delta _{2}}$.

Now, 
\begin{equation*}
P_{n}(\Delta _{1})=\sum_{\substack{ \sigma \in \Lambda _{n}  \\ S_{\sigma
}[0,1]\supseteq \Delta _{1}}}p_{\sigma }=\sum_{\substack{ \sigma \in \Lambda
_{n-\lambda }  \\ S_{\sigma }[0,1]\supseteq \widehat{\Delta _{1}}}}p_{\sigma
}\sum_{\substack{ \omega \text{ path generation }\lambda  \\ S_{\sigma
\omega }[0,1]\supseteq \Delta _{1}}}p_{\omega }.
\end{equation*}%
Let $\mathcal{D}_{1}$ denote those $\sigma \in \Lambda _{n-\lambda }$ such
that $S_{\sigma }[0,1]$ contains $\widehat{\Delta _{1}}$, but intersects $%
\widehat{\Delta _{2}}$ only at the endpoint, and $\mathcal{E}$ denote the
remaining $\sigma \in \Lambda _{n-\lambda }$ such that $S_{\sigma }[0,1]$
contains $\widehat{\Delta _{1}}$. Define $\mathcal{D}_{2}$ similarly. With
this notation we have

\begin{eqnarray*}
P_{n}(\Delta _{1}) &=&\sum_{\sigma \in \mathcal{D}_{1}%
}p_{\sigma }\sum_{\substack{ \omega \text{ path generation }\lambda  \\ %
S_{\sigma \omega }[0,1]\supseteq \Delta _{1}}}p_{\omega } +
\sum_{\sigma \in \mathcal{E}%
}p_{\sigma }\sum_{\substack{ \omega \text{ path generation }\lambda  \\ %
S_{\sigma \omega }[0,1]\supseteq \Delta _{1}}}p_{\omega } \\
&\leq &\sum_{\sigma \in \mathcal{D}_{1}}p_{\sigma }\Gamma_{\Delta_1,\lambda}
+\sum_{\sigma \in \mathcal{E}}p_{\sigma } \\
&\leq &B(\lambda )R_{\lambda }\left( \sum_{\sigma \in \mathcal{D}%
_{1}}p_{\sigma }+\sum_{\sigma \in \mathcal{E}}p_{\sigma }\right) +\left(
\sum_{\sigma \in \mathcal{D}_{2}}p_{\sigma }+\sum_{\sigma \in \mathcal{E}%
}p_{\sigma }\right) \\
&\leq &B(\lambda )R_{\lambda }P_{n-\lambda }(\widehat{\Delta _{1}}%
)+P_{n-\lambda }(\widehat{\Delta _{2}})\text{.}
\end{eqnarray*}%
By the induction assumption, $P_{n-\lambda }(\widehat{\Delta _{1}})\leq
cq^{n-\lambda }P_{n-\lambda }(\widehat{\Delta _{2}})$. Combining this with
the inequalities $B(\lambda )\leq q^{\lambda }/2$ and $c\geq 2/R_{\lambda },$
we obtain%
\begin{eqnarray*}
P_{n}(\Delta _{1}) &\leq &(cq^{n-\lambda }B(\lambda )R_{\lambda
}+1)P_{n-\lambda }(\widehat{\Delta _{2}}) \\
&\leq &cq^{n}R_{\lambda }P_{n-\lambda }(\widehat{\Delta _{2}})\leq
cq^{n}P_{n}(\Delta _{2}),
\end{eqnarray*}%
where the last inequality comes from the definition of $R_{\lambda }$.

Symmetric reasoning gives the other inequality in equation \eqref{ProveReg}.
\end{proof}

\begin{remark}
It follows that for any generalized regular self-similar measure, the formulas for the
local dimensions stated in Theorem \ref{CompFormula} hold.
\end{remark}

\begin{remark}
A measure can be comparable, but not generalized regular. One example is the self-similar measure associated with the IFS $S_0(x)=x/3$, $S_1(x)=x/3+1/3$, $S_2(x)=x/3+2/3$, with probabilities $p_0=p_2=2/5$, $p_1=1/5$.
\end{remark}

\subsection{The Structure of the set of local dimensions}

From Lemma \ref{lem:children}, we see that the characteristic vector of the children
    of $\Delta$ depend only on the characteristic vector of $\Delta$.
As such, we can construct a finite directed graph of characteristic vectors, 
    where we have a directed edge from $\gamma$ to $\beta$ if there is a $\Delta$ 
    with characterstic vector $\gamma$ and a child of $\Delta$ with a characterstic
    vector $\beta$.
This graph is known as the \textit{transition graph }and the paths in the graph, the sequences $(\gamma
_{1},...,\gamma _{J})$ where each $\gamma_{j}$ is a child of $\gamma
_{j-1}, $ are called \textit{admissible paths}.

As in \cite{HHM}, a non-empty subset $\Omega ^{\prime }$ of the set of
characteristic vectors $\Omega $ is called a \textit{loop class }if whenever 
$\sigma ,\beta \in \Omega ^{\prime }$, then there is an admissible path $%
(\gamma _{1},...,\gamma _{J})$ of characteristic vectors $\gamma_i \in \Omega^\prime$ 
    such that $\alpha =\gamma _{1}$ and $\beta =\gamma _{J}$. These are the strongly connected
components of the transition graph. A loop class is called an \textit{%
essential class} if, in addition, whenever $\alpha \in \Omega ^{\prime }$
and $\beta \in \Omega $ is a child of $\alpha $, then $\beta \in \Omega
^{\prime }$. In graph theory terminology, an essential class is a 
strongly connected component that does not have a path going to a vertex outside 
of the component.  These will always exist for finite directed graphs.

In the equicontractive case, Feng in \cite{F2} showed that there was always
a unique essential class. The same is true in the non-equicontractive case,
with a similar argument - it can be shown that there is some $\gamma \in
\Omega $ such that for every $\alpha \in \Omega $ there is an admissible
path beginning with $\alpha $ and ending with $\gamma $. The maximal loop
class containing such a $\gamma $ is the essential class. To produce such a $%
\gamma $ we take the characteristic vector $\gamma =\mathcal{C}_{n}(\Delta
)=(\ell _{n}(\Delta ),V_{n}(\Delta ),r_{n}(\Delta ))$, where $\ell
_{n}(\Delta )$ is minimal, and among all characteristic vectors of minimal
length (in this sense) we take a choice for $\gamma $ where $V_{n}(\Delta )$
is maximal in cardinality. Checking that this choice of $\gamma $ has the
desired properties is similar to Feng's argument, using the slightly revised
definitions. The details are left to the reader.

If $x\in K$ has symbolic representation $(\gamma _{1},\gamma _{2},...)$ and
there is some $N$ so that $\gamma _{j}\in \Omega ^{\prime }$ for all $j\geq
N $ for some loop class $\Omega ^{\prime }$, we say the point $x$ is \textit{%
in the loop class }$\Omega ^{\prime }$. Points in the essential class are
also called \textit{essential points}. Just as shown in \cite[Proposition 3.6]{HHN}%
, the set of essential points has full $\mu $-measure and full Hausdorff $%
H^{s}$-measure when $s$ is the Hausdorff dimension of $K$.

We will say the loop class $\Omega ^{\prime }$ is of \textit{positive type}
if there is an admissible path $\eta $ in $\Omega ^{\prime }$ such that $%
T(\eta )$ has all positive entries. When $K=[0,1]$ the essential class is of
positive type. The proof is the same as \cite[Proposition 4.12]{HHM}.

The main results of \cite[Section 5]{HHM} extend to our setting, with the
same proofs.

\begin{theorem}
Suppose the IFS is of finite type, the self-similar measure $\mu $ is
comparable and $\Omega ^{\prime }$ is a loop class of positive type. The set
of local dimensions of $\mu $ at the points in $\Omega ^{\prime }$ is a
closed interval and the local dimensions of the periodic points in $\Omega
^{\prime }$ are dense in that interval.
\end{theorem}

\begin{corollary}
\label{DimEss}If the IFS is of finite type, the self-similar set $K=[0,1]$
and the self-similar measure $\mu $ is comparable, then the set of local
dimensions at the points in the essential class is a closed interval.
\end{corollary}

\begin{proof}
We have already noted that the assumption that $K=[0,1]$ ensures the
essential class is of positive type.
\end{proof}

\begin{remark}
These results can be partially extended to the case that $\mu $ is not comparable. We
refer the reader to \cite[Section 3]{HHN} for the technical details of how
this is done in the equicontractive case. Similar arguments apply here.
\end{remark}

In many cases, the essential class is either $[0,1]$ or $(0,1),$ so that the
set of local dimensions of the measure is either a closed interval or the
union of a closed interval and one or two isolated points. This is the case,
for instance, in the examples studied in the next section.

\section{Examples}
\label{sec:examples}

\subsection{Oriented Bernoulli convolutions}

In this subsection we will compare the structure and local dimensions of the
(biased) Bernoulli convolutions with contraction factor the inverse of the golden mean,
with that of the self-similar measures where one contraction is oriented in
the opposite direction.

Let 
\begin{align*}
S_0 (x) & = rx & R_0(x) & =  r-rx \\
S_1(x) & = rx + 1-r & R_1(x) & = 1-rx
\end{align*}
where $r = \frac{\sqrt{5}-1}{2} \approx 0.618$ is the inverse of the golden mean 
    satisfying the relation $r^2+r-1=0$. 
We will investigate the self-similar measure $\mu =\mu (p_{0},p_{1})$ arising
from the IFS $\{S_{0},S_{1}\}$ with probabilities $\{p_{0},p_{1}\}$, which we
will denote by $SS,$ and the self-similar measure $\nu =\nu (p_{0},p_{1})$
from the IFS $\{S_{0},R_{1}\}$, denoted $SR$. Both IFS will generate $[0,1]$
as its self-similar set.

The analysis of the self-similar measures arising from the IFS $\{R_0, S_1\}$ is similar to that coming from the $\{S_0, R_1\}$ case, and hence is omitted.
The self-similar measure arising from the IFS $\{R_0, R_1\}$ is generalized regular in 
    the case $p_0 = p_1 = 1/2$ and in that case, the self-similar measure is identical to 
    that from the three other IFS $\{R_0, S_1\}$, $\{S_0, R_1\}$ and $\{S_0, S_1\}$.
The analysis of the biased $\{R_0, R_1\}$ case is left to the interested reader. We note that the essential class is again $(0,1)$.

It follows from Proposition \ref{reg} that $\mu $ is generalized regular if  $%
p_{0}=p_{1}$ and $\nu $ is generalized regular if $p_{0}\leq p_{1}$. The
multifractal analysis of the measure $\mu $ in the regular case was
investigated in \cite{F1}, where it was shown that the set of attainable local dimensions was the interval%
\begin{equation*}
\left[ \frac{\log 1/2}{\log r}-\frac{1}{2},\frac{\log 1/2}{\log r}\right].
\end{equation*}%
For any choice of $p_0 < p_1$ (the non-regular case), it was proven in \cite{HHN} that there is always an isolated point in the set of local dimensions of $\mu$.

The two measures coincide if $p_{0}=p_{1}$, so we will assume otherwise. We
will show the following about these measures.

\begin{theorem}
\label{thm:golden}
For any $p_{0}< 1/2 < p_{1}$, 
\begin{equation*}
\{\dim _{loc}\mu (x):x\in \lbrack 0,1]\}=[a,b]\bigcup \left\{ \frac{\log
p_{0}}{\log r}\right\}
\end{equation*}%
and%
\begin{equation*}
\{\dim _{loc}\nu (x):x\in \lbrack 0,1]\}=[A,B]\bigcup \left\{ \frac{\log
p_{0}}{\log r}\right\} 
\end{equation*}
where $a, A, b$ and $B$ are summarized below:
\[ 
\arraycolsep=1.4pt\def\arraystretch{2.2}
\begin{array}{|l|c|c|}
\hline
\mathrm{Range}  &   a & b \\
\hline
0 < p_0 \leq 1/3    &   a = \frac{\log p_1}{\log r}                                    &                             \\  \cline{1-2}
1/3 < p_0 \leq s  & \frac{\log 2 p_0 p_1}{2 \log r} \leq a < \frac{\log{p_1}}{\log r} & b = \frac{\log p_0 p_1}{2 \log r}   \\ \cline{1-2}
s < p_0 < 1/2 & \frac{\log 2 p_0 p_1}{2 \log r} \leq a \leq 
\frac{\log \frac{p_0 p_1}{2} \left(1- p_0 p_1 + \sqrt {(1+ p_0 p_1) (1-3 p_0 p_1)}\right)}{4\log r}  &  \\ 
\hline
\end{array}
\]
where $1/3 < s = \frac{1}{2} - \frac{\sqrt{6 \sqrt{13}-21}}{6} < r^2$, and 
\[ 
\arraycolsep=1.4pt\def\arraystretch{2.2}
\begin{array}{|l|c|c|}
\hline
\mathrm{Range}  &   A & B \\
\hline
0 < p_0 \leq r^2 & A = \frac{\log{p_1}}{\log r}       & \frac{\log p_0 p_1/r}{2 \log r} \leq B  \leq \frac{\log p_0}{2 \log r}  \\ \hline
r^2 < p_0 < 1/2 & \frac{\log{p_0}}{2 \log r} \leq A \leq \frac{\log p_0 p_1 /r}{2 \log r}  & B  =  \frac{\log p_1}{\log r}  \\
\hline
\end{array}.
\]

In particular: 
\begin{enumerate}
\item $B < b$; $a<b$;
\label{it:1}
\item If $1/3 < p_0 \leq r^2$ then $a<A$; If $p_0 \le 1/3$, then $a=A$;
\label{it:3}

\item If $p_0 \neq r^2$, then $A < B$;
\label{it:5}
\item If $p_{0}=r^{2}$, then $A = B = 1$ and $\dim _{loc}\nu (x)=\{1,2\}$.  Moreover, $\nu$ is an absolutely continuous measure with respect to Lebesgue measure and has density function

\[
f(x)=\left\{ 
\begin{array}{cc}
\frac{2x}{r} & \text{if }0\leq x\leq r \\ 
\frac{2(1-x)}{r^2} & \text{if }r\leq x\leq 1%
\end{array}%
\right. .
\]	

\label{it:4}

\end{enumerate}
\end{theorem}

In contrast, it follows from \cite{LNR} that regardless of the choice of probabilities, the measure $\mu$ is purely singular to Lebesgue measure.

We remark that $s$ is the choice of $p_0$ where $sp(T_2) = \sqrt{sp(T_2T_3)}$, with the matrices, $T_2$ and $T_3$, defined below.
Note that if $p_{0}=p_{1}^{2}$, then $p_{0}=r^{2}=(3-\sqrt{5})/2$.

In Figure \ref{fig:lc}(a) we give the upper and lower bounds for $a$ and $b$,
    and in (b) we give the upper and lower bounds for $A$ and $B$.

\begin{figure}[tbp]
\begin{center}
\subfigure[SS Case]{\includegraphics[scale=0.5,angle=270]{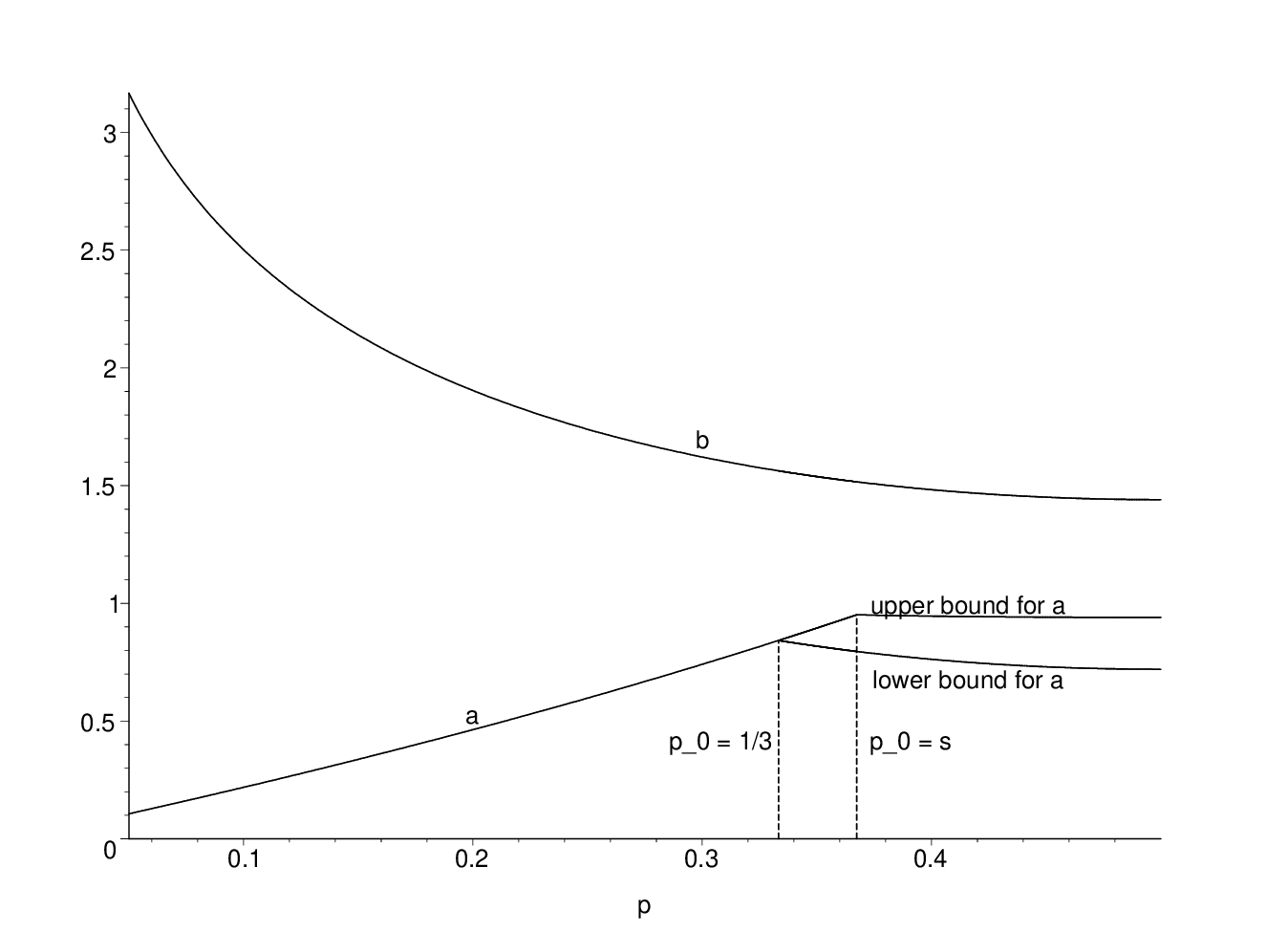}} %
\subfigure[SR Case]{\includegraphics[scale=0.5,angle=270]{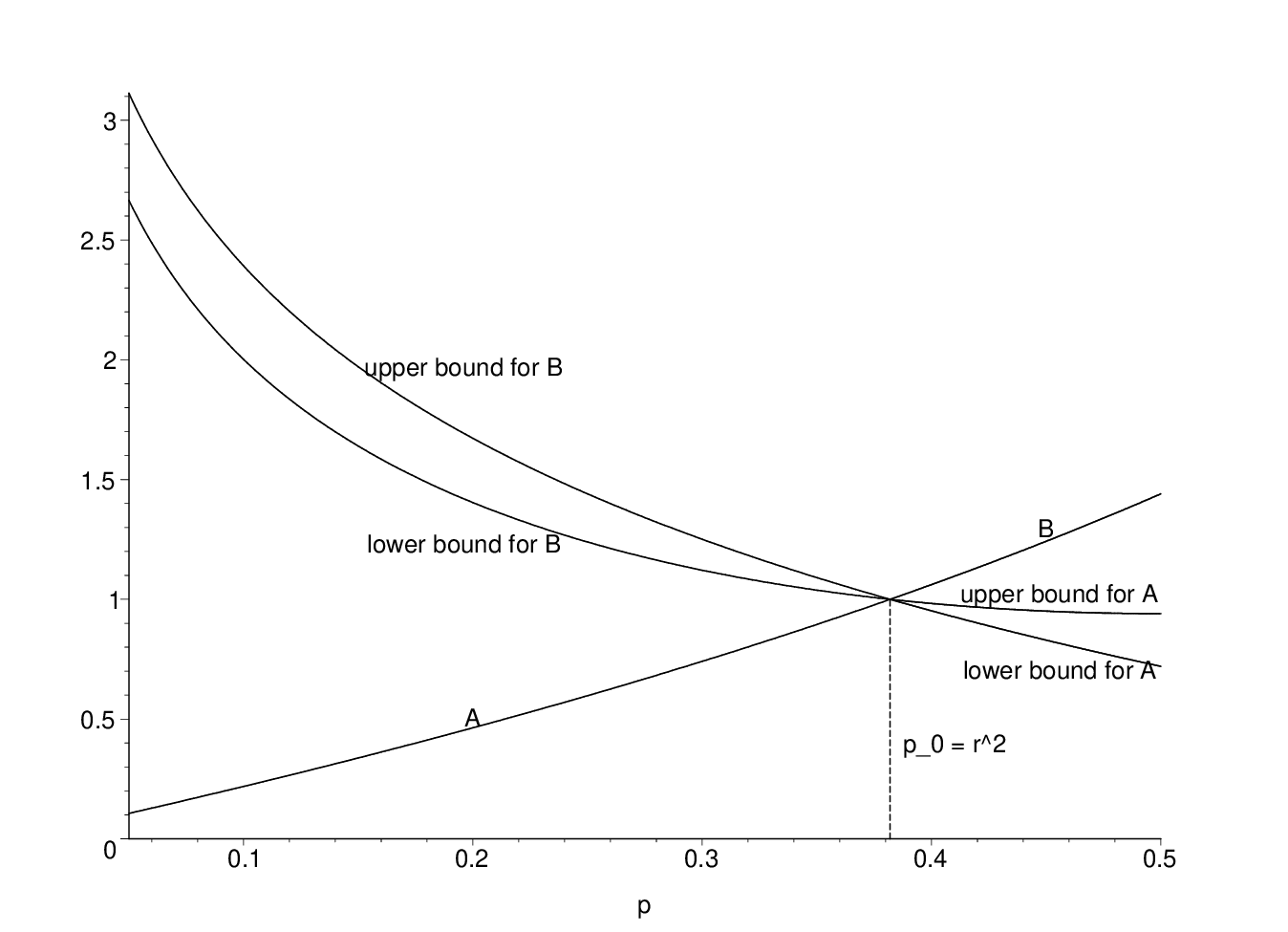}}
\end{center}
\caption{Upper and lower bounds for the range of local dimensions of the
essential class}
\label{fig:lc}
\end{figure}

The proof depends on the transition matrices, so we begin by giving
information about the characteristic vectors and by listing the important
transition matrices for the two measures.

For the IFS$\ SS$ this information can be found in \cite{F1}, but we repeat
the main facts here for the convenience of the reader. There are 7
characteristic vectors, denoted $\{1,2,...,7\},$ with children: $1\rightarrow
2,3,4$; $2\rightarrow 2,3$; $4\rightarrow 3,4$; $3,7\rightarrow 5$; $%
5\rightarrow 3,6,7$ and $6\rightarrow 3$. See Figure \ref{fig:golden} (a) for the transition diagram. The
essential characteristic vectors are $\{3,5,6,7\}$ and the open interval $%
(0,1)$ is the set of essential points. The endpoints $0$ and $1$ have
symbolic representations $(1,2,2,...)$ and $(1,4,4,...)$ respectively. It
can be shown that $T(2,2)=p_{0}$, $T(4,4)=p_{1}$.

Any loop in the essential class is a composition of the 
\begin{equation} \label{Loop}
\text{Loops: }5\rightarrow 6\rightarrow 3\rightarrow 5;5\rightarrow
7\rightarrow 5\text{ and }5\rightarrow 3\rightarrow 5.
\end{equation}%
The corresponding transition matrices are:%
\begin{equation*}
T(5,6,3,5)=%
\begin{bmatrix}
p_{0}^{2}p_{1} & p_{0}^{2}p_{1} \\ 
p_{0}p_{1}^{2} & p_{0}p_{1}^{2}%
\end{bmatrix}%
:=T_{1}
\end{equation*}

\begin{equation*}
T(5,7,5)=%
\begin{bmatrix}
p_{0}p_{1} & p_{0}p_{1} \\ 
0 & p_{1}^{2}%
\end{bmatrix}%
:=T_{2}
\end{equation*}%
\begin{equation*}
T(5,3,5)=%
\begin{bmatrix}
p_{0}^{2} & 0 \\ 
p_{0}p_{1} & p_{0}p_{1}%
\end{bmatrix}%
:=T_{3}.
\end{equation*}

For the IFS $SR$ there are 8 characteristic vectors, with the only
difference being that $5\rightarrow 7,6,8$ and $8\rightarrow 5$. See Figure %
\ref{fig:golden} (b) for the transition diagram. The essential characteristic vectors are $\{3,5,6,7,8\}$
and the set of essential points is again $(0,1)$. The symbolic
representations of $0,1$ are the same as before, and it is still true that $%
T(2,2)=p_{0}$. However, $T(4,4)=p_{0}$. The paths in the essential class are
compositions of the loops $5\rightarrow 6\rightarrow 3\rightarrow 5$; $%
5\rightarrow 7\rightarrow 5$ and $5\rightarrow 8\rightarrow 5$, with
transition matrices%
\begin{equation*}
T(5,6,3,5)=%
\begin{bmatrix}
p_{0}p_{1}^{2} & p_{0}^{2}p_{1} \\ 
p_{0}^{2}p_{1} & p_{0}p_{1}^{2}%
\end{bmatrix}%
:=T_{1}^{\prime }
\end{equation*}

\begin{equation*}
T(5,7,5)=%
\begin{bmatrix}
p_{1}^{2} & p_{0}p_{1} \\ 
0 & p_{0}^{2}%
\end{bmatrix}%
:=T_{2}^{\prime }
\end{equation*}%
\begin{equation*}
T(5,8,5)=%
\begin{bmatrix}
p_{0}^{2} & 0 \\ 
p_{0}p_{1} & p_{1}^{2}%
\end{bmatrix}%
:=T_{3}^{\prime }.
\end{equation*}

\begin{figure}[tbp]
\begin{center}
\subfigure[Transition graph for
SS]{\includegraphics[scale=0.65]{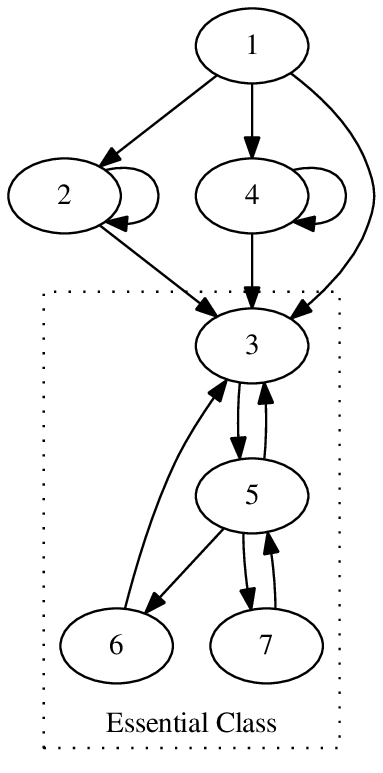}} 
\subfigure[Transition graph
for SR]{\includegraphics[scale=0.65]{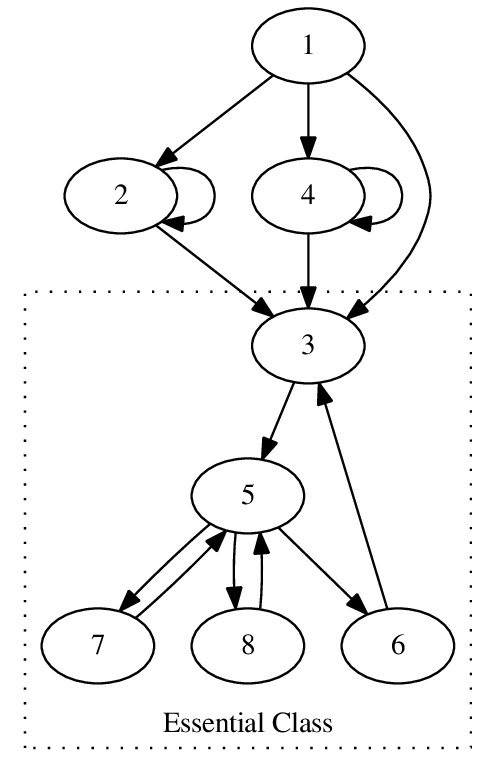}}
\end{center}
\caption{Non-reduced transition graph for SS and SR}
\label{fig:golden}
\end{figure}

Before turning to the proof of the theorem we record the following
elementary fact that was also used in \cite{HHN}.

\begin{lemma}
\label{lem:pseudonorm}
Let $T$ $=(T_{ij})$ be any transition matrix. Denote by $\left\Vert
T\right\Vert _{\min,c}$ and $\left\Vert T\right\Vert _{\max,c}$ the
(pseudo)-norms%
\begin{equation*}
\left\Vert T\right\Vert_{\min,c}=\min_{j}\sum_{i}\left\vert
T_{ij}\right\vert ,\text{ }\left\Vert T\right\Vert_{\max,
c}=\max_{j}\sum_{i}\left\vert T_{ij}\right\vert ,
\end{equation*}%
the minimal and maximal column sums of $T$. Then 
\begin{equation*}
\left\Vert T_{1}T_{2}\right\Vert _{\min,c}\geq \left\Vert T_{1}\right\Vert
_{\min,c}\left\Vert T_{2}\right\Vert _{\min,c}\text{, }\left\Vert
T_{1}T_{2}\right\Vert _{\max,c}\leq \left\Vert T_{1}\right\Vert _{\max,c}
\left\Vert T_{2}\right\Vert _{\max,c}
\end{equation*}%
and%
\begin{equation*}
\left\Vert T\right\Vert _{\min,c}\leq \left\Vert T\right\Vert \leq
N\left\Vert T\right\Vert _{\max,c}
\end{equation*}%
where $N$ is the number of columns of $T$. The analogous statement holds for
minimal and maximal row sums, denoted by $\Vert \cdot \Vert_{\min, r}$ and $%
\Vert \cdot \Vert_{\max, r}$ respectively.
\end{lemma}

\begin{proof}[Proof of Theorem \protect\ref{thm:golden}]
As $p_{0}<p_{1}$, Prop. \ref{reg} ensures that $\nu $ (but not $\mu$) is a generalized regular
measure. 

As the essential class is $(0,1)$ in both cases, the set of local dimensions
consists of a closed interval together with the local dimensions at $0$ and $%
1;$ see Corollary \ref{DimEss} for the generalized regular, non-equicontractive case and \cite%
[Corollary 3.15]{HHN} for the non-regular, equicontractive case. For 
    both $\mu$ and $\nu$ we have $[0]=(1,2,2,...)$.
This gives us that
\begin{equation*}
\dim _{loc}\mu (0)=\frac{\log p_{0}}{\log r}=\dim _{loc}\nu (0).
\end{equation*}%
Similarly, $[1]=(1,4,4,...)$ and $T(4,4)=p_{0}$ for the measure $\nu ,$ so
we also have $\dim _{loc}\nu (1)=\log p_{0}/\log r$. This proves that the
set of local dimensions of $\nu $ is the union of a closed interval and the
point $\log p_{0}/\log r$.

However, $\dim _{loc}\mu (1)=\log p_{1}/\log r$. We will next show that this
local dimension is also attained at an essential point of the IFS\ $SS$.
Consider an essential, periodic, boundary point $z$ with
symbolic representations given by period $(5,7,5)$ and hence also by period $%
(5,3,5)$. Applying Proposition \ref{periodic} shows that 
\begin{equation*}
\dim _{loc}\mu (z)=\frac{\log sp(T_{2})}{2\log r}=\frac{\log p_{1}}{\log r}.
\end{equation*}%
Thus, $\dim _{loc}\mu (1)=\dim _{loc}\mu (z)\in \lbrack a,b]$. That
completes the proof of the first statement in the theorem. \newline
 
\noindent \textbf{Case 1: Bounds for $a$} \newline
Since all loops in the essential class of the IFS\ $SS$ are compositions of
the three loops listed in equation \eqref{Loop}, the norm of any transition matrix from the
essential class is comparable to the norm of a suitable product $%
T=T_{{i_{1}}}T_{i_{2}}\cdot \cdot \cdot T_{i_{j}},$ where $i_{l}$ $\in
\{1,2,3\}$. Consequently, appealing to Lemma \ref{lem:pseudonorm} $$\left\Vert T\right\Vert
^{1/n}\leq 2^{1/n}\max \left( \Vert T_{1}\Vert _{\max ,r}^{1/3},\Vert
T_{2}\Vert _{\max ,r}^{1/2},\Vert T_{3}\Vert _{\max ,r}^{1/2}\right) $$ where $%
n$ is the length of the path associated with $T$. But this maximum is $\Vert
T_{2}\Vert _{\max ,r}^{1/2}$, thus for any $x$ in the essential
class $(0,1)$,%
\begin{eqnarray*}
\dim _{loc}\mu (x) &\geq &\frac{\log \left( \max \left( \Vert T_{1}\Vert
_{\max ,r}^{1/3},\Vert T_{2}\Vert _{\max r}^{1/2},\Vert T_{3}\Vert _{\max
,r}^{1/2}\right) \right) }{\log r} \\
&=&\frac{\log \Vert T_{2}\Vert _{\max ,r}^{1/2}}{\log r}=\frac{\log \left(
\max (2p_{0}p_{1},p_{1}^{2})\right) }{2\log r}.
\end{eqnarray*}%
An easy calculation shows that 
\begin{equation*}
a=\min_{x\in (0,1)}\dim _{loc}\mu (x)\geq 
\begin{cases}
\frac{\log p_{1}}{\log r} & 0<p_{0}\leq 1/3 \\ 
\frac{\log \left( 2p_{1}p_{0}\right) }{2\log r} & 1/3\leq p_{0}<1/2%
\end{cases}%
.
\end{equation*}

Of course, the local dimension at any point in the essential class is an
upper bound for $a$. In particular, $a\leq \min_{j}\dim _{loc}\mu (x_{j})$
where we choose for $x_{j}$, $j=1,2,3,4$, the periodic points with periods
given by the cycles associated with $T_{j}$, $j=1,2,3$ and $T_{2}T_{3}$ for $%
j=4$. An easy calculation gives
\begin{equation*}
\dim _{loc}\mu (x_{j})=%
\begin{cases}
\frac{\log sp(T_{j})}{d_{j}\log r} & \text{for }j=1,2,3\text{ where }%
d_{1}=3,d_{2}=d_{3}=2 \\ 
\frac{\log sp(T_{2}T_{3})}{4\log r} & \text{for }j=4%
\end{cases}%
.
\end{equation*}%
It can be checked that $\min_{j}\dim _{loc}\mu (x_{j})$ arises with $j=2$
when $0\leq p_{0}\leq s$ and with $j=4$ when $s\leq p_{0}<1/2$. Thus 
\begin{eqnarray*}
a &\leq &\min_{j\in \{1,2,3,4\}}\dim _{loc}\mu (x_{j})=%
\begin{cases}
\frac{\log sp(T_{2})}{2\log r} & \text{if }0\leq p_{0}\leq s \\ 
\frac{\log sp(T_{2}T_{3})}{4\log r} & \text{if }j=4%
\end{cases}
\\
& = &%
\begin{cases}
\frac{\log p_{1}}{\log r} & \text{if }0\leq p_{0}\leq s \\ 
\frac{\log \frac{p_{0}p_{1}}{2}\left( 1-p_{0}p_{1}+\sqrt{%
(1+p_{0}p_{1})(1-3 p_{0}p_{1})}\right) }{%
4\log r} & \text{if }s\leq p_{0}<1/2%
\end{cases}%
,
\end{eqnarray*}%
as claimed.

We next show that when $1/3<p_{0}\leq r^{2}$, then $a$ is
strictly less than $\frac{\log {p_{1}}}{2\log r}$. For $s<p_{0}\leq r^{2}$
this is clear from the above. 
An easy induction argument shows that 
\begin{equation*}
T_{2}^{k}=%
\begin{bmatrix}
(p_{0}p_{1})^{k} & \sum_{j=1}^{k}p_{0}^{j}p_{1}^{2k-j} \\ 
0 & p_{1}^{2k}%
\end{bmatrix}%
,
\end{equation*}%
thus%
\begin{equation*}
T_{1}T_{2}^{k}=%
\begin{bmatrix}
p_{0}^{k+2}p_{1}^{k+1} & \sum_{j=2}^{k+2}p_{0}^{j}p_{1}^{2k+3-j} \\ 
p_{0}^{k+1}p_{1}^{k+2} & \sum_{j=1}^{k+1}p_{0}^{j}p_{1}^{2k+3-j}%
\end{bmatrix}%
.
\end{equation*}%
It is straightforward to see that 
\begin{eqnarray*}
sp(T_{1}T_{2}^{k}) &=&\sum_{j=1}^{k+2}p_{0}^{j}p_{1}^{2k+3-j}=\frac{p_{0}}{%
p_{1}}\left( \frac{1-\left( \frac{p_{0}}{p_{1}}\right) ^{k+2}}{1-\frac{p_{0}%
}{p_{1}}}\right) p_{1}^{2k+3} \\
&=&\left( 1-\left( \frac{p_{0}}{p_{1}}\right) ^{k+2}\right) \frac{%
p_{0}p_{1}^{2k+3}}{p_{1}-p_{0}}.
\end{eqnarray*}

When $p_0>1/3$, then  $p_{0}>p_{1}-p_{0}$, so for sufficiently large (but fixed) $k$ we have $%
sp(T_{1}T_{2}^{k})>p_{1}^{2k+3}$. Consequently, any periodic essential point 
$z_{k}$ with period $(5,6,3,(5,7)^{k})$ has 
\begin{equation*}
\dim _{loc}\mu (z_{k})\leq \frac{\log sp(T_{1}T_{2}^{k})}{(2k+3)\log r}<%
\frac{\log p_{1}}{\log r}.
\end{equation*}%
Thus $a\leq \min_{k}\dim _{loc}\mu (z_{k})<\log p_{1}/\log r.$ \newline

\noindent \textbf{Case 2: Bounds for $b$} \newline
We similar know that if $T=T_{{i_{1}}}T_{i_{2}}\cdot \cdot \cdot T_{i_{j}},$
where $i_{l}$ $\in \{1,2,3\}$, then $\left\Vert T\right\Vert ^{1/n}\geq \min
\left( \Vert T_{1}\Vert _{\min ,c}^{1/3},\Vert T_{2}\Vert _{\min
,c}^{1/2},\Vert T_{3}\Vert _{\min ,c}^{1/2}\right) $ when $n$ is the length
of the path associated with $T$. But this minimum is $\Vert T_{3}\Vert
_{\min ,c}^{1/2}$ (for all choices of $p_{0}$), consequently, 
\begin{align*}
b& =\max_{x\in (0,1)}\dim _{loc}\mu (x) \\
& \leq \frac{\log \left( \Vert T_{3}\Vert _{\min ,c}^{1/2}\right) }{\log r}=%
\frac{\log p_{0}p_{1}}{2\log r}.
\end{align*}%
Furthermore, $b\geq \max_{j}\dim _{loc}\mu (x_{j})$ where we choose $x_{j}$
as above. The maximum occurs at $x_{3}$, thus 
\begin{equation*}
b\geq \frac{\log sp(T_{3})}{\log r}=\frac{\log p_{0}p_{1}}{2\log r}
\end{equation*}%
and that proves the equality holds. \newline

\noindent \textbf{Case 3: Bounds for $A$} \newline
The arguments are similar for $A$, but using the transition matrices $%
T_{1}^{\prime },$ $T_{2}^{\prime },T_{3}^{\prime }$ and the analogous points 
$x_{j}^{\prime }$. Thus 
\begin{align*}
A& =\min_{x\in (0,1)}\dim _{loc}\nu (x) \\
& \geq \log \left( \max \left( \Vert T_{1}^{\prime }\Vert _{\max
,c}^{1/3},\Vert T_{2}^{\prime }\Vert _{\max c}^{1/2},\Vert T_{3}^{\prime
}\Vert _{\max ,c}^{1/2}\right) \right) /\log r \\
& =\frac{\log \Vert T_{2}^{\prime }\Vert _{\max ,c}^{1/2}}{\log r}=\frac{%
\log (\max \left( p_{0},p_{1}^{2}\right) )}{2\log r}
\end{align*}%
and therefore%
\begin{equation*}
A\geq 
\begin{cases}
\frac{\log p_{1}}{\log r} & \text{if }p_{0}\leq p_{1}^{2}\text{, i.e., }%
0<p_{0}\leq r^{2} \\ 
\frac{\log p_{0}}{2\log r} & \text{if }p_{0}\geq p_{1}^{2}\text{, i.e., }%
r^{2}\leq p_{0}<1/2%
\end{cases}%
.
\end{equation*}%
Further, 
\begin{align*}
A& \leq \min_{j\in \{1,2,3,4\}}\dim _{loc}\nu (x_{j}^{\prime }) \\
& \leq \log \left( \max \left( sp(T_{1}^{\prime })^{1/3},sp(T_{2}^{\prime
})^{1/2},sp(T_{3}^{\prime })^{1/2},sp(T_{2}^{\prime }T_{3}^{\prime
})^{1/4}\right) \right) /\log (r) \\
& =%
\begin{cases}
\frac{\log sp(T_{2}^{\prime })^{1/2}}{\log r} & \text{if }0<p_{0}\leq r^{2}
\\ 
\frac{\log sp(T_{2}^{\prime }T_{3}^{\prime })^{1/4}}{\log r} & \text{if }%
r^{2}\leq p_{0}<1/2%
\end{cases}%
\end{align*}%
and consequently, as claimed, 
\begin{equation*}
A\leq 
\begin{cases}
\frac{\log p_{1}}{\log r} & \text{if }0<p_{0}\leq r^{2} \\ 
\frac{\log p_{0}p_{1}/r}{2\log r} & \text{if }r^{2}\leq p_{0}<1/2%
\end{cases}%
.
\end{equation*}

\noindent \textbf{Case 4: Bounds for $B$} \newline
As with $b$, we have 
\begin{align*}
B& \leq \log \left( \min \left( \Vert T_{1}^{\prime }\Vert _{\min
,c}^{1/3},\Vert T_{2}^{\prime }\Vert _{\min ,c}^{1/2},\Vert T_{3}^{\prime
}\Vert _{\min ,c}^{1/2}\right) \right) /\log r \\
& =\log \left( \Vert T_{3}^{\prime }\Vert _{\min ,c}^{1/2}\right) /\log r \\
& =\log \min \left( p_{0},p_{1}^{2}\right) /2\log r \\
& =%
\begin{cases}
\frac{\log p_{0}}{2\log r} & \text{if }0<p_{0}\leq r^{2} \\ 
\frac{\log p_{1}}{\log r} & \text{if }r^{2}\leq p_{0}<1/2%
\end{cases}%
\end{align*}%
and also 
\begin{align*}
B& \geq \log \left( \min \left( sp(T_{1}^{\prime })^{1/3},sp(T_{2}^{\prime
})^{1/2},sp(T_{3}^{\prime })^{1/2},sp(T_{2}^{\prime }T_{3}^{\prime
})^{1/4}\right) \right) /\log r \\
& =%
\begin{cases}
\frac{\log sp(T_{2}^{\prime }T_{3}^{\prime })^{1/4}}{\log r} & \text{if }%
0<p_{0}\leq r^{2} \\ 
\frac{\log sp(T_{2}^{\prime })^{1/2}}{\log r} & \text{if }r^{2}\leq p_{0}<1/2%
\end{cases}
\\
& =%
\begin{cases}
\frac{\log (p_{0}p_{1}/r)}{2\log r} & \text{if }0<p_{0}\leq r^{2} \\ 
\frac{\log p_{1}}{\log r} & \text{if }r^{2}\leq p_{0}<1/2%
\end{cases}%
.
\end{align*}

Finally, to see that $\nu $ is an absolutely continuous measure when $p_{0}=r^{2}$,  with density function 
$f$ defined in (4), we note that the absolutely continuous measure $\sigma $ given by $\sigma
(A)=\int_{A}f$ is a probability measure that satisfies the self-similarity
equation $\sigma (E)=p_{0}\sigma \circ S_{o}^{-1}(E)+p_{1}\sigma \circ
R_{1}^{-1}(E)$ for all intervals $E$. It follows by uniqueness that $\sigma
=\nu $.

\end{proof}

\begin{remark}
It would be interesting to know if $a<A$ for all $p_{0}\in (1/3, 1/2)$.
\end{remark}

Similar results can be obtained for oriented Bernoulli convolutions
associated with the other simple Pisot numbers. For the $SS$ case, we refer
the reader to \cite{HHN}.

\subsection{Other Examples}

\begin{example}
Consider the IFS generated by the contractions%
\begin{eqnarray*}
S_{i}(x) &=&\frac{x}{N}+\frac{i}{N}\text{ for }i=0,...,N-1, \\
S_{N}(x) &=&\frac{x}{N^{2}}+\frac{i_{0}}{N}
\end{eqnarray*}%
where $i_{0}$ is chosen from $\{1,...,N-2\}$. If, for example, $N=3$ and $%
i_{0}=1$, there are three reduced characteristic vectors 
\begin{eqnarray*}
1 &:&=(1,(0,1)) \\
2 &:&=(1/3,(0,1),(0,1/3)) \\
3 &:&=(2/3,(-1/3,1))
\end{eqnarray*}%
with transitions $1\rightarrow 123112311231$; $2\rightarrow 1231$; $%
3\rightarrow 12311231$. Thus the essential class is the full self-similar
set $[0,1]$. Similar conclusions hold for the general case. If $\mu $ is the
associated self-similar measure given by probabilities with $%
p_{0}=p_{N-1}=\min p_{j}$, then $\mu $ is generalized regular, so that the set of all
local dimensions of $\mu $ is a closed interval.
\end{example}

\begin{example} \label{ExReg}
Consider the IFS 
\begin{equation*}
S_{0}(x)=\frac{x}{2}\text{, }S_{1}(x)=\frac{x}{2}+\frac{1}{4}\text{, }%
S_{2}(x)=\frac{x}{2}+\frac{1}{2}
\end{equation*}%
with probabilities $p_{0}=p_{2}:=p,$ $p_{1}=1-2p$. The associated self-similar measure is regular if and only if $p \leq 1/3$. 
We claim this equicontractive example is also generalized regular if $p \geq 1/3$. Thus our new definition of
generalized regularity covers equicontractive examples that are not regular.

We will also see that the set of attainable local dimensions is a closed interval whenever $p_0 \ge 1/3$, but admits an isolated point (the local dimension at $0$) when $p_0 <1/3$.

There are four reduced characteristic vectors, $1:=(1,(0))$, $2:=(1/2,(0))$, 
$3:=(1/2,(0,1/2))$, $4:=(1/2,(1/2))$, with transitions $1\rightarrow 2,3a,3b,4$%
; $2\rightarrow 2,3$; $3\rightarrow 3a,3b$; $4\rightarrow 3,4$. The only
essential vector is $3$ and the essential class is the subset $(0,1)$. The
transition matrices are%
\begin{eqnarray*}
T(1,2) &=&[p]\text{, }T(1,3a)=[1-2p,p]\text{, }T(1,3b)=[p,1-2p]\text{, }%
T(1,4)=[p] \\
T(2,2) &=&[p]\text{, }T(2,3)=[1-2p,p]\text{, }T(4,4)=[p]\text{, }%
T(4,3)=[1-2p,p]
\end{eqnarray*}%
and 
\begin{equation*}
A:=T(3,3a)=%
\begin{bmatrix}
p & 0 \\ 
p & 1-2p%
\end{bmatrix}%
\text{, }B:=T(3,3b)=%
\begin{bmatrix}
1-2p & p \\ 
0 & p%
\end{bmatrix}%
.
\end{equation*}%
To establish generalized regularity we must study the ratios%
\begin{equation*}
\frac{P_{n+m}(\Delta ^{\prime })}{P_{m}(\Delta )}\sim \frac{\left\Vert
T(\gamma _{0},\gamma _{1},...,\gamma _{m+n})\right\Vert }{\left\Vert
T(\gamma _{0},\gamma _{1},...,\gamma _{m})\right\Vert }
\end{equation*}%
over all $\Delta ^{\prime }\subseteq \Delta ,$ net intervals of generations $m+n$
and $m$ respectively.

If $\Delta ^{\prime }=(1,2,...,2)$ or $(1,4,...,4)$, then it is clear that
this ratio is at least $p^{n}$. So assume otherwise. Without loss of
generality, $\Delta ^{\prime }=(1,2,...,2,3,...,3)$ so that%
\begin{equation*}
T(\Delta ^{\prime
})=T(1,2,...,2,3,...,3)=T(1,2)T(2,2)^{a}T(2,3)A^{k_{1}}B^{j_{1}}\cdot \cdot
\cdot A^{k_N}B^{j_{N}}
\end{equation*}%
where $a+2+\sum_{i=1}^{N}(k_{i}+j_{i})=m+n$, $k_{1},j_{N}\geq 0$ and $%
k_{i},j_{i}>0$ otherwise. We will also assume%
\begin{equation*}
T(\Delta )=T(1,2)T(2,2)^{a}T(2,3)A^{k_{1}}B^{j_{1}}\cdot \cdot \cdot
B^{j_{M}}
\end{equation*}%
where $M\leq N,$ $a+2+\sum_{i=1}^{M}(k_{i}+j_{i})=m$, $k_{1},j_{M}\geq 0$ and $%
k_{i},j_{i}>0$ otherwise. (If, instead, $\Delta =(1,2,...,2)$, the arguments are
similar.) As $T(1,2)$, $T(2,2)$ and $T(2,3)$ have all positive entries, it
will suffice to study%
\begin{equation*}
\frac{\left\Vert A^{k_{1}}B^{j_{1}}\cdot \cdot \cdot
B^{j_{M}}A^{k_{M+1}}\cdot \cdot \cdot B^{j_{N}}\right\Vert }{\left\Vert
A^{k_{1}}B^{j_{1}}\cdot \cdot \cdot B^{j_{M}}\right\Vert }.
\end{equation*}%
We will give the details for the case $j_{i},k_{i}>0$, as the
other cases are similar. Put $\varepsilon =1-2p \le p$. Recall $p\ge 1/3$, hence $\varepsilon >0$.

An easy induction argument shows%
\begin{equation*}
A^{k}=%
\begin{bmatrix}
p^{k} & 0 \\ 
p^{k}(1+C(k)\varepsilon ) & \varepsilon ^{k}%
\end{bmatrix}%
\text{, }B^{j}=%
\begin{bmatrix}
\varepsilon ^{j} & p^{j}(1+D(j)\varepsilon ) \\ 
0 & p^{j}%
\end{bmatrix}%
\end{equation*}%
where $C(k),D(j)\geq 0$. Consequently 
\begin{equation*}
A^{k}B^{j}=%
\begin{bmatrix}
p^{k}\varepsilon ^{j} & p^{j+k}(1+D(j)\varepsilon ) \\ 
p^{k}\varepsilon ^{j}(1+C(k)\varepsilon ) & p^{j+k}(1+G^{\prime
}(j,k)\varepsilon )%
\end{bmatrix}%
\end{equation*}%
where $(1+G^{\prime }(j,k)\epsilon) \ge (1+D(j)\epsilon)(1+C(k)\epsilon)\geq 0$. Moreover, as $\varepsilon \leq p$, we have  $%
p^{k}\varepsilon ^{j}\leq p^{j+k}(1+D(j)\varepsilon )$ and $p^{k}\varepsilon
^{j}(1+C(k)\varepsilon )\leq p^{j+k}(1+G^{\prime }(j,k)\varepsilon )$.

Note that when you multiply together $2\times 2$ matrices, with the first
column entries dominated by the second column entries, we get another matrix
with the same property. Thus further computation shows that 
\begin{equation*}
A^{k_{1}}B^{j_{1}}\cdot \cdot \cdot B^{j_{M}}=%
\begin{bmatrix}
e(m) & p^{m}(1+G_{m}\varepsilon ) \\ 
f(m) & p^{m}(1+G_{m}^{\prime }\varepsilon )%
\end{bmatrix}%
\end{equation*}%
where again $e(m),f(m),G_{m},G_{m}^{\prime }\geq 0$ (and really depend on
the tuple, $(k_{1},j_{1},...,j_{M}),$ not just their sum) with the first
column dominated by the second column. We obtain a similar formula for $%
A^{k_{M+1}}\cdot \cdot \cdot B^{j_{N}}$ and therefore%
\begin{equation*}
A^{k_{1}}B^{j_{1}}\cdot \cdot \cdot B^{j_{M}}A^{k_{M+1}}\cdot \cdot \cdot
B^{j_{N}}=%
\begin{bmatrix}
e(m,n) & p^{m+n}(1+G_{m}\varepsilon )(1+G_{n}^{\prime }\varepsilon
)+e(m)p^{n}(1+G_{n}\varepsilon ) \\ 
f(m,n) & p^{m+n}(1+G_{m}^{\prime }\varepsilon )(1+G_{n}^{\prime }\varepsilon
)+f(m)p^{n}(1+G_{n}\varepsilon )%
\end{bmatrix}%
\end{equation*}%
(again, where the first column entries are dominated by the second column).
It follows that 
\begin{eqnarray*}
\frac{\left\Vert A^{k_{1}}B^{j_{1}}\cdot \cdot \cdot
B^{j_{M}}A^{k_{M+1}}\cdot \cdot \cdot B^{j_{N}}\right\Vert }{\left\Vert
A^{k_{1}}B^{j_{1}}\cdot \cdot \cdot B^{j_{M}}\right\Vert } &\geq &\frac{%
\text{sum of entries of col 2}}{2\cdot \text{sum of entries of col 2}} \\
&\geq &\frac{p^{m+n}(2+G_{m}\varepsilon +G_{m}^{\prime }\varepsilon
)(1+G_{n}^{\prime }\varepsilon )}{2p^{m}(2+G_{m}\varepsilon +G_{m}^{\prime
}\varepsilon )}\geq \frac{p^{n}}{2}.
\end{eqnarray*}

On the other hand, the only edge paths of length $n$ are
 $(0)^{n}$ or $(2)^{n}$ and thus for any $\Delta $, $%
\Gamma _{\Delta ,n}\leq 2p^{n}$. It follows that $B(n)$ is bounded and that
proves the generalized regularity property. (See \ref{regnotation} and \ref{regdefn} for the meaning of this notation and definition of generalized regular.)

We note, moreover, that as the symbolic representation of $0$ is $(1,2,2,2,\dots)$, we have $dim_{loc} \mu (0) = \log p / \log (1/2)$ and this is also the local dimension at the essential point with symbolic representation $(1,3,3,\cdots)$. Consequently, the set of local dimensions is a closed interval when $p \ge 1/3$.

When $p<1/3$, a minimum column sum argument shows that $$\log \mu(x) \le \frac{\log(\min(2p,1-2p))}{\log(1/2)}$$ for all $x \in (0,1)$ and hence the local dimension at $0$ is isolated.
\end{example}

\end{document}